\documentclass[techreport]{stellatoarticle}  

\title{Second-Order Switching Time Optimization for Switched Dynamical Systems}

\author{Bartolomeo~Stellato, Sina Ober-Bl\"{o}baum~and~Paul J.\ Goulart
\ifTechReport\else\thanks{This work was supported by the People Programme (Marie Curie Actions) of the European Union’s Seventh Framework Programme (FP7/2007-2013) under REA grant agreement no 607957 (TEMPO).}
\fi
\thanks{The authors are with the University of Oxford, Parks Road, Oxford, OX1 3PJ, U.K. (e-mail: {\texttt\{bartolomeo.stellato, sina.ober-blobaum, \mbox{paul.goulart\}@eng.ox.ac.uk}}).}
}

\begin{document}

  \maketitle

  \begin{abstract}
    \ifTwoColumn \boldmath \else \fi
    Switching time optimization arises in finite-horizon optimal control for switched systems where, given a sequence of continuous dynamics, one minimizes a cost function with respect to the switching times. We propose an efficient method for computing the optimal switching times for switched linear and nonlinear systems. A novel second-order optimization algorithm is introduced where, at each iteration, the dynamics are linearized over an underlying time grid to compute the cost function, the gradient and the Hessian efficiently. With the proposed method, the most expensive operations at each iteration are shared between the cost function and its derivatives, thereby greatly reducing the computational burden. We implemented the algorithm in the Julia package \texttt{SwitchTimeOpt} allowing the user to easily solve switching time optimization problems. In the case of linear dynamics, many operations can be further simplified and benchmarks show that our approach is able to provide optimal solutions in just a few $\mathrm{ms}$. In the case of nonlinear dynamics, two examples show that our method provides optimal solutions with up to two orders of magnitude time reductions over state-of-the-art approaches.
  \end{abstract}


  \ifTwoColumn \else
  \ifTechReport \else
  \newpage
  \fi
  \fi


\section{Introduction}
\label{sec:Introduction}

Hybrid dynamical models commonly arise in several engineering problems where systems with continuous dynamics interact with discrete events. The analysis and control of such systems lies at the boundary between control engineering (system dynamics) on the one hand and computer science (discrete events) on the other. Hybrid systems have been used to model a wide variety of control problems in fields such as process control, automotive industry, power systems, traffic control and many others.

Optimal control of hybrid systems presents several challenges because it involves both continuous and discrete decisions which make the corresponding optimization problems $NP$-hard to solve~\cite{Belotti:2013dv} in general.
Although there have been several major advances in recent years in the solution of hybrid systems optimal control problems~\cite{Sager:R92r0gWi,Anonymous:6Ys-F2zR}, there are still many open questions regarding the quality of the solutions, the efficiency of the algorithms and the computation speed when dealing with fast dynamical systems.


Switched systems are a particular class of hybrid systems consisting of several continuous subsystems where a switching law defines the active one at each time instant. A recent survey on computational methods for switched systems control appears in~\cite{Zhu:2014dj}.
In the present work we focus on optimal control of autonomous switched systems where the sequence of continuous dynamics is fixed. In particular, we study the problem of computing the optimal switching instants at which the ordered dynamics must change in order to minimize a given cost function. This problem is usually referred to as \emph{switching time optimization}.

This topic has been studied extensively in the last decade. In~\cite{Seatzu:2006jz} the authors provide a method to construct an offline mapping of the optimal switching times for linear dynamics from the initial state of the system.
Even though this approach seems appealing at first sight, it suffers from the high storage requirements typical for explicit control approaches \cite{Bemporad:2002hfa} as the dimension of the system and the number of possible switchings increase.


More recent approaches focus on finding optimal switching times using iterative optimization methods. In~\cite{Egerstedt:2006hm} an expression for the gradient of the cost function with respect to the switching times is derived for the case of general nonlinear systems. A first-order method based on Armijo step sizes is then adopted to find the optimal switching times. An extension for discrete-time nonlinear systems is given in~\cite{Flasskamp:2013ta}. However, first order methods are very sensitive to the problem data and can exhibit slow convergence~\cite{nocedal2006numerical}. In~\cite{Johnson:2011dta} an expression for the Hessian of the cost function is derived for nonlinear dynamics and a second-order method is adopted to find the optimal switching times showing significant improvements on the number of iterations compared to the first-order method in~\cite{Egerstedt:2006hm}. However, both these first and second-order approaches suffer from the computational complexity of multiple numerical integrations required to solve the differential equations used to define the cost function, the gradient and the Hessian (in the second-order case). Note that the Hessian definition in~\cite{Johnson:2011dta} requires an additional set of integrations to be performed.

In the literature, there has been very limited focus on the computational effort required by the switching time optimization and the multiple integration routines. In~\cite{Ding:2009hla} the authors present a convergence analysis of a second-order method for switched nonlinear systems similar to the one in~\cite{Johnson:2011dta} without considering the overall computation time. In~\cite{Caldwell:2012ei} the switching time optimization problem for linear time-varying dynamics is formulated so that only a set of differential equations needs to be solved before the optimization procedure. Once the integration is performed, the steepest descent direction can be computed directly without solving any further differential equations. However, in~\cite{Caldwell:2012ei} no closed-form expression for the Hessian of the cost function is provided and only a steepest descent algorithm is adopted.

In this work we present a novel method to solve switching time optimization problems efficiently for linear and nonlinear dynamics. 
We develop efficiently computable expressions for the cost function, the gradient and the Hessian, exploiting shared terms in the most expensive computations. In this way, at each iteration of the optimization algorithm there is no significant increase in complexity in computing the gradient or the Hessian once the cost function is evaluated.  These easily computable expressions are obtained thanks to linearizations of the system dynamics around equally spaced grid points, and then integrated via independent matrix exponentials. In the case of linear dynamics, our method can be greatly simplified and the matrix exponentials decomposed into independent scalar exponentials that can be parallelized to further reduce the computation times.

Our method has been implemented in the open-source Julia package \texttt{SwitchTimeOpt}~\cite{SwitchTimeOpt} with a simple interface that allows the user to easily define and solve switching time optimization problems. Through Julia's \texttt{MathProgBase} interface, \texttt{SwitchTimeOpt} supports a wide variety of nonlinear solvers which can be quickly interchanged.

We provide three examples to benchmark the performance of our method. 
The first, from~\cite{Caldwell:2012ei}, is a system with two unstable switched dynamics whose optimal switching times are obtained in few $\unit[]{ms}$ with our approach. The second one is the so-called Lotka-Volterra fishing problem~\cite{Volterra:1926ws} with nonlinear dynamics, integer control inputs and constant steady state values to be tracked. The third is a double-tank system first appeared in~\cite{Malmborg:1997tb} and used in switching time optimization setting in~\cite{Axelsson:2005dd}. In the final example we apply our algorithm to find the optimal switching times to track a time-varying reference level of the liquid in one of the tanks. In both the nonlinear examples our method, which is implemented on a high-level language, shows up to two orders of magnitude improvements over tailored state-of-the-art nonlinear optimal control software tools.

The rest of the paper is organized as follows. Section~\ref{sec:Problem Statement} defines the switching time optimization problem in terms of switching time intervals. In Section~\ref{sec:Preliminaries} we define the preliminary notions we used to obtain the main result: the linearization grid, the solution algorithm and the necessary definitions. In Section~\ref{sec:Main Result} we present the main result and describe the numerical computations required to obtain the cost function, the gradient and the Hessian. In Section~\ref{sec:Linear Switched Systems} we propose some simplifications and precomputations for our algorithm in the case of linear dynamics. In Section~\ref{sec:Examples} we describe our open-source Julia package and describe its application to several example problems. Conclusions are provided in Section~\ref{sec:Conclusion}.

\section{Problem Statement}
\label{sec:Problem Statement}

Consider a switched autonomous system switching between $N$ modes, the dynamics for which can be expressed as
\begin{equation}\label{eq:nonlin_system_dynamics}
  \dot{x}(t) = f_i(x(t)), \quad t\in [\tau_i, \tau_{i+1}),\quad i=0,\dots,N,
\end{equation}
with each $f_i : \mathbb{R}^n \to \mathbb{R}^n$ assumed differentiable and the state  $x(t)\in \mathbb{R}^{n_x}$ assumed to have initial state $x(0) = x_0$.
We will refer to the times $\tau_i$ as the \emph{switching times}, and define also the \emph{switching intervals}
\[
\delta_i := \tau_{i+1}-\tau_i, \quad i=0,\dots,N,
\]
so that each $
  \tau_i = \sum_{j=0}^{j-1}\delta_j 
$. 
In the sequel, we will take the set of switching intervals $\delta := \{\delta_i\}_{i=0}^N$ as decision variables to be optimized, but will occasionally use the switching times $\tau:=\{\tau_i\}_{i=0}^{N+1}$ for convenience of notation.  We define the final time as $T_\delta := \sum_i\delta_i = \tau_{N+1}$, with initial time $\tau_0 = 0$.

Our goal is to find optimal switching intervals $\delta^*$ minimizing an objective function in Bolza form
\begin{equation}\label{eq:intro_cost_function}
  \underbrace{x(T_\delta)^\top \bar{E} x(T_\delta)}_{=:\psi(\delta)} + \underbrace{\int_{0}^{T_\delta} x(t)^\top \bar{Q} x(t) \mathrm{d}t}_{=:\mathcal{L}(\delta)}.
\end{equation}
The Mayer term $\psi$ penalizes the final state at time $T_{\delta}$ with weights defined by matrix $\bar{E} = \bar{E}^\top \in \mathbb{S}_{+}^{n_x}$. The Lagrange term $\mathcal{L}$ penalizes the integral between $0$ and $T_\delta$ of the quadratic state penalty weighted by the  symmetric positive semidefinite matrix $\bar{Q} = \bar{Q}^\top \in \mathbb{S}_{+}^{n_x}$.

We include a set of constraints on the switching intervals
\ifTwoColumn
  \begin{equation*}
    \begin{multlined}[t][\columnwidth]\Delta = \Big\{\delta \in \mathbb{R}_{+}^{N+1}\;\Big|\\ 0\leq \underline{b}_i \leq \delta_i \leq \bar{b}_i\; i=0,\dots,N\; \wedge\; T_\delta = T \Big\},\end{multlined}
  \end{equation*}
\else
  \begin{equation*}
    \Delta = \left\{\delta \in \mathbb{R}_{+}^{N+1}\;\middle|\; 0\leq \underline{b}_i \leq \delta_i \leq \bar{b}_i\; i=0,\dots,N\; \wedge\; T_\delta = T \right\},
  \end{equation*}
\fi
which requires all switching times to be nonnegative and the final time $T_\delta$ to be equal to some desired final time $T$. In addition, in case the $i$-th dynamics must be active for a minimum or maximum time, we allow lower and upper bounds $\underline{b}_i$ and $\bar{b}_i$ respectively. If neither minimum nor maximum constraints are imposed for interval $\delta_i$, we set $\underline{b}_i = 0$ and $\bar{b}_i = \infty$.

Our complete switching time optimization problem then takes the form
\begin{equation}\label{eq:original_sto_problem}
\begin{aligned}
&\underset{\delta}{\text{minimize}} && \int_{0}^{T_\delta} x(t) ^\top \bar{Q} x(t) \mathrm{d} t + x(T_\delta)^\top \bar{E} x(T_\delta) \\
&\text{subject to} &&  \dot{x}(t) = f_i(x(t)), \quad t\in [\tau_{i}, \tau_{i+1}),\quad i=0,\dots,N\\
&&& x(0) = x_0\\
&&&\delta \in \Delta.
\end{aligned}\tag{$\mathcal{P}$}
\end{equation}

Although we restrict ourselves to the case where the switching order of the $N+1$ modes is prespecified, we allow the system dynamics to be identical for different $i$.  If we set some $\delta_i=0$, then the $i^{\textrm{th}}$ interval collapses and the dynamics switch directly from the $(i-1)^{\textrm{th}}$ to the $(i+1)^{\textrm{th}}$ mode.
This allows some dynamics to be bypassed and an arbitrary switching order realized without recourse to integer optimization; see~\cite{Gerdts:2006hr}. For example, given $N_{\mathrm{dyn}}$ different dynamics, one can cycle through them in the same predefined ordering $N_{\mathrm{dyn}}$ times for a total of $N_{\mathrm{dyn}}^2$ dynamics and $N_{\mathrm{dyn}}^2 - 1$ switching times, thereby allowing the dynamics to be visited in arbitrary order.  We illustrate the use of this approach in the examples in Section \ref{sec:Examples}.




The cost function in problem~\eqref{eq:original_sto_problem} is non-convex in general, but it is smooth~\cite{Ding:2009hla} and its first and second derivatives can be used efficiently within a nonlinear optimization method, e.g.\ sequential quadratic programming (SQP) or interior point (IP) method to obtain locally optimal switching times.
In order to obtain a real-time implementable algorithm, we derive tractable formulations of the cost function, the gradient and the Hessian based on linearizations of the system dynamics.  We will show that this approach offers significant improvements in computational efficiency relative to competing approaches in the literature.

\section{Preliminaries}
\label{sec:Preliminaries}


\subsection{Time Grid and Dynamics Linearization}
\label{sub:Time Grid and Dynamics Linearization}

%


In order to integrate the switched nonlinear dynamics~\eqref{eq:nonlin_system_dynamics}, we define an equally spaced ``background'' grid of $n_{\mathrm{grid}}$ time-points from $0$ to the final time $T$, and hold these background grid points fixed regardless of the choice of switching times $\tau_i$.  Note that,  
depending on the intervals $\delta$, the switching times $\tau$ can be in different positions relative to the background grid while maintaining the ordering $\tau_{i} \leq \tau_{i+1}$.

We subdivide each interval $\delta_i$ according to the background grid points falling between $\tau_{i}$ and $\tau_{i+1}$, with $\tau_i^j$ denoting the $j^\text{th}$ grid point after the switching time $\tau_i$.   The number of such background grid  points between switching times $\tau_i$ and $\tau_{i+1}$ is denoted by $n_i$, which is itself a function of the switching times $\tau$.  Note that 
we set $n_{N+1} = 0$.

For notational convenience, we will define $\tau_i^0 := \tau_i$ and $\tau_i^{n_i+1} := \tau_{i+1}$ for $i=0,\dots,N$.  We further define a partitioning of the switching intervals such that $\delta_i^j$ is the $j^{\text{th}}$ subdivision of interval $i$, so that \begin{equation}\label{eq:deltai_partition}
  \delta_i = \sum_{j=0}^{n_i}\delta_i^j \quad \text{and} \quad   \tau_i^{k} = \tau_i + \sum_{j=0}^{k-1}\delta_i^j,\quad \forall k\leq n_i.
\end{equation}
In subsequent sections we will define a number of vector and matrix quantities to be associated with the time instants $\tau_i^j$, and will adopt complementary notation, e.g.\ vector $x_i^j$ and matrix $M_i^j$ are associated with the $j^{\text{th}}$ grid point after switching time $\tau_i$.  Likewise $x_i^{n_i + 1} := x_{i+1}, x_i^0 := x_i$ and $M_i^{n_i + 1} := M_{i+1}, M_i^0 := M_i$.

A portion of the grid is presented in Figure~\ref{fig:time_grid} where the smaller ticks represent the background grid points.
%
%
%

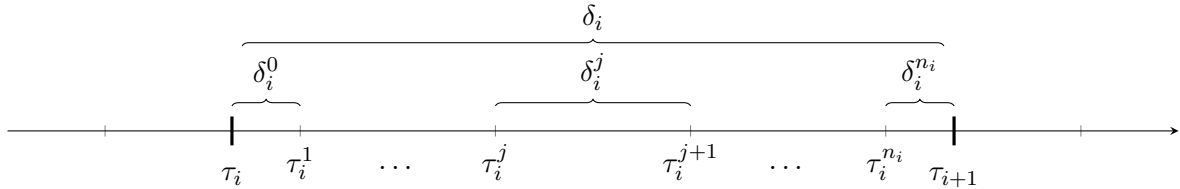
\begin{figure}[h]
  \centering


  \ifTwoColumn
    \begin{tikzpicture}
      \begin{axis}[
        width=1.2\columnwidth,
        hide y axis,
        axis x line = middle,
        axis y line = left,
        enlarge x limits = 0.025,
        ymin = -2.0, ymax = 2.0,
        xmin = 0, xmax = 10,
        xtick ={0,2,4,6,8,10},
        xticklabel style={align=center, text height = 2ex},
        xticklabels={,$\tau_{i}^{1}$,$\tau_{i}^{j}$,$\tau_{i}^{j+1}$,$\tau_{i}^{n_i}$,},
        extra x ticks={1.3,8.7},
        extra tick style={xticklabels={$\tau_i$,
        $\tau_{i+1}$},tick style={very thick,black},         xticklabel style={align=center}, major tick length=1em}]

  \else
    \begin{tikzpicture}
      \begin{axis}[
        width=1.1\columnwidth,
        hide y axis,
        axis x line = middle,
        axis y line = left,
        enlarge x limits = 0.1,
        ymin = -2.0, ymax = 2,
        xmin = 0, xmax = 10,
        xtick ={0,2,4,6,8,10},
        xticklabel style={align=center, text height = 2ex},
        xticklabels={,$\tau_{i}^{1}$,$\tau_{i}^{j}$,$\tau_{i}^{j+1}$,$\tau_{i}^{n_i}$,},
        extra x ticks={1.3,8.7},
        extra tick style={xticklabels={$\tau_i$,
        $\tau_{i+1}$},tick style={very thick,black},         xticklabel style={align=center}, major tick length=1em}]
    \fi

        \ifTwoColumn
        \node[anchor = north] at (3.0,-0.15) {$\dots$};
        \node[anchor = north] at (7.0,-0.15) {$\dots$};
        \else
        \node[anchor = north] at (3.0,-0.1) {$\dots$};
        \node[anchor = north] at (7.0,-0.1) {$\dots$};
        \fi

        \node[anchor = south] (taui) at (1.3,0){};
        \node[anchor = south] (tauip) at (8.7,0){};
        \draw [decoration={brace, raise= 6ex}, decorate] (taui) -- node [anchor = south, yshift = 6.5ex, pos = 0.5]{$\delta_i$}(tauip);

        \draw [decoration={brace, raise= 2ex}, decorate] (1.3,0) -- node [anchor = south, yshift = 2.5ex, pos = 0.5]{$\delta_i^0$}(2.0, 0);
        \draw [decoration={brace, raise= 2ex}, decorate] (4,0) -- node [anchor = south, yshift = 2.5ex, pos = 0.5]{$\delta_i^j$}(6.0, 0);
        \draw [decoration={brace, raise= 2ex}, decorate] (8.0,0) -- node [anchor = south, yshift = 2.5ex, pos = 0.5]{$\delta_{i}^{n_i}$}(8.7, 0);

    \end{axis}
    \end{tikzpicture}
    	\caption{Switching times within the time grid.}
    	\label{fig:time_grid}
  \end{figure}

In order to make the computations of the cost function and its derivatives numerically efficient, we linearize the dynamics around each time instant of the background grid and each switching time.

For a given time instant $\tau_i^j$ we consider the linearized dynamics around the state ${x_{i}^{j} \coloneqq x(\tau_i^{j})}$ with $j=0,\dots,n_i+1$ (to simplify the notation we consider $x_i = x_i^0 = x(\tau_i)$) by writing
\begin{equation}
  \begin{aligned}
    \dot{x}(t) &\approx f_i(x_i^j) + J_{f_i}(x_i^j)(x(t) - x_i^j)\\
    & = J_{f_i}(x_i^j)x(t) + \Big(f_i(x_i^j) - J_{f_i}(x_i^j)x_i^j\Big)
  \end{aligned}
\end{equation}
where
\begin{equation}\label{eq:jacobian_nonlinear_dynamics}
  J_{f_i}(x_i^j)= \left.\frac{\partial f_i}{\partial x}\right|_{x_i^j = {x(\tau_i^j)}}
\end{equation}
is the Jacobian of the $i^{\textrm{th}}$ nonlinear dynamics evaluated at $x_i^j$.
We can obtain an approximate linear model by augmenting the dynamics with an additional constant state so that
\begin{equation}\label{eq:linearized_dynamics_sto}
  \dot{x}(t)=A_{i}^{j}x(t), \quad t\in [\tau_i^j, \tau_i^{j+1}),
\end{equation}
where
  \begin{equation}\label{eq:linearized_dynamics_sto1}
  A_{i}^{j} = \begin{bmatrix}J_{f_i}(x_i^j) & f_i(x_i^j) - J_{f_i}(x_i^j)x_i^j\\ 0 & 0\end{bmatrix}
  \end{equation}
  and $x(t)$ is an augmented version of the previous state definition, i.e.~$x(t) \mapsfrom \begin{bmatrix}x(t),1\end{bmatrix}^\top$.


\subsection{Solution Approach}
\label{sub:Solution Approach}

Once the dynamics are linearized as in~\eqref{eq:linearized_dynamics_sto}, and after defining the augmented cost function weights $Q \coloneqq \mathrm{blkdiag}(\bar{Q}, 0)$ and $E \coloneqq \mathrm{blkdiag}(\bar{E}, 0)$, we can approximate the problem~\eqref{eq:original_sto_problem} as
\ifTwoColumn
  \begin{equation}\label{eq:linearized_sto_problem}
  \begin{aligned}
  &\underset{\delta}{\text{minimize}} && \int_{0}^{T_\delta} x(t) ^\top Q x(t) \mathrm{d} t + x(T_\delta)^\top E x(T_\delta) \\
  &\text{subject to} &&  \begin{multlined}[t][.6\columnwidth]\dot{x}(t) = A_i^j x(t), \quad t\in [\tau_i^j, \tau_i^{j+1}),\\ \; j=0,\dots,n_i,\quad i=0,\dots,N\end{multlined}\\
  &&& x(0) = x_0\\
  &&&\delta \in \Delta.
  \end{aligned}
  \tag{$\mathcal{P}_{\mathrm{lin}}$}
  \end{equation}
\else
  \begin{equation}\label{eq:linearized_sto_problem}
  \begin{aligned}
  &\underset{\delta}{\text{minimize}} && \int_{0}^{T_\delta} x(t) ^\top Q x(t) \mathrm{d} t + x(T_\delta)^\top E x(T_\delta) \\
  &\text{subject to} &&  \dot{x}(t) = A_i^j x(t), \quad t\in [\tau_i^j, \tau_i^{j+1}),\quad i=0,\dots,N, \quad j=0,\dots,n_i\\
  &&& x(0) = x_0\\
  &&&\delta \in \Delta.
  \end{aligned}
  \tag{$\mathcal{P}_{\mathrm{lin}}$}
  \end{equation}
\fi

We will make use of problem~\eqref{eq:linearized_sto_problem} to approximate the original problem~\eqref{eq:original_sto_problem} at each iteration of a standard second-order nonlinear programming routine such as IPOPT~\cite{Wachter:2006hk}. By linearizing of the system dynamics around the state trajectory, we can directly construct problem~\eqref{eq:linearized_sto_problem}.

In the remainder of the paper we focus on the numerical evaluation of the cost function $J(\delta)$, the gradient $\nabla J(\delta)$ and the Hessian $H_J(\delta)$ for problem~\eqref{eq:linearized_sto_problem} which can be computed efficiently.

A prototype algorithm is sketched in Algorithm~\ref{alg:solve_sto_problem}. Note that in line~\ref{alg:lin:linearize_sto_problem} the act of linearizing problem~\eqref{eq:original_sto_problem} produces the majority of the computational work with the benefit that the cost function and its derivatives can be then computed efficiently in line~\ref{alg:lin:compute_J_and_derivs}.

%
%
%
\begin{algorithm}
\caption{Solve Switching Time Optimization Problem~\eqref{eq:original_sto_problem}}
\label{alg:solve_sto_problem}
\begin{algorithmic}[1]
\Function{SwitchingTimeOptimization}{}
  \While{Termination conditions not met}
    \State Linearize problem~\eqref{eq:original_sto_problem} \label{alg:lin:linearize_sto_problem}
    \State Compute $J(\delta)$, $\nabla J (\delta)$ and $H_J(\delta)$ for~\eqref{eq:linearized_sto_problem} \label{alg:lin:compute_J_and_derivs}
    \State Perform one NLP solver iteration obtaining $\delta^{(k+1)}$
  \EndWhile
\EndFunction
\end{algorithmic}
\end{algorithm}


\subsection{Definitions}
\label{sub:Definitions}

We next present some preliminary definitions required to develop our main result.


\begin{definition}[State evolution]\label{def:state_evolution}
  The matrix $\Phi(t, \tau_i^j)$ is the state transition matrix of the linearized system from $\tau_i^j$ to $t$, and is defined as:
  \begin{equation}\label{eq:state_trans_mat}
    \Phi(t, \tau_i^j) \coloneqq  e^{A^m_l(t - \tau_l^m)}\prod_{p=0}^{m-1}e^{A_{l}^p\delta_l^p}\prod_{q=i+1}^{l-1}\prod_{p=0}^{n_q}e^{A_q^p\delta_q^p}\prod_{p=j}^{n_{i}}e^{A_{i}^{p}\delta_{i}^p},
  \end{equation}
  where $\tau_l$ and $\tau_l^m$ are the last switching time and the last grid point before $t$ respectively.
\end{definition}
Given a time instant $\tau_i^j$ and a time $t\in \mathbb{R}_{+}$ such that $t\geq \tau_i^j$ we can define the state $x(t)$ as
\begin{equation}
  x(t) = \Phi(t,\tau_i^j)x_i^j.
\end{equation}
Observe that if we consider transition between two switching times $\tau_i$ and $\tau_j$ with $\tau_i \leq \tau_j$, the state transition matrix in~\eqref{eq:state_trans_mat} simplifies to
\begin{equation}
    \Phi(\tau_{l}, \tau_i) = \prod_{q=i}^{l-1}\prod_{p=0}^{n_q}e^{A_q^p\delta_q^p},
\end{equation}
which will be used extensively in most of the computations in the remainder of the paper.

\begin{definition}[Cost-to-go matrices]\label{def:matrices_P_F_S}
  Given the time $\tau_i^j$, define matrix $P_{i}^j\in\mathbb{S}_+^{n_x}$ as
  \begin{equation}\label{eq:def_P}
    P_{i}^j \coloneqq \int_{\tau_i^j}^{T_\delta}\Phi(t,\tau_i^j)^\top Q \Phi(t, \tau_i^j)\mathrm{d}t,
  \end{equation}
  where $\Phi(t,\tau_i^j)$ is the state transition matrix in Definition~\ref{def:state_evolution}.  Define the matrix $F_{i}\in\mathbb{S}_+^{n_x}$ as
  \begin{equation}\label{eq:def_F}
    F_{i}^j \coloneqq \Phi(T_\delta, \tau_i^j)^\top E  \Phi(T_\delta, \tau_i^j).
  \end{equation}
  Define the sum of these two matrices as
  \begin{equation}\label{eq:S_sum_P_F}
    S_{i}^j \coloneqq P_{i}^j + F_{i}^j, \quad i=0,\dots,N+1.
  \end{equation}
\end{definition}
  Following the convention described in Section~\ref{sub:Time Grid and Dynamics Linearization}, we will denote $P_i^0 = P_i, F_i^0=F_i, S_i^0 = S_i$ and $P_i^{n_i+1} = P_{i+1}, F_{i}^{n_i+1} = F_{i+1}, S_{i}^{n_i +1}= S_{i+1}$.

\begin{definition}[Matrix $C$]\label{def:C_matrix_sto}
  Given matrices $S_{i}$ with $i=0,\dots,N+1$ and $A_i^{n_i}$ with $i=0,\dots,N$, define matrices $C_{i} \in \mathbb{S}_{+}^{n_x}$ as
  \begin{equation}\label{eq:C_matrix_sto}
    C_{i} = Q + \left(A_i^{n_i}\right)^\top S_{i+1} + S_{i+1} A_i^{n_i}, \quad i=0,\dots,N.
  \end{equation}
\end{definition}

\section{Numerical Solution Method}
\label{sec:Main Result}

We are now in the position to derive the cost function and its first and second derivatives for Problem~\eqref{eq:linearized_sto_problem}.

\begin{theorem}[Cost function $J(\delta)$, gradient $\nabla J(\delta)$ and Hessian $H_J(\delta)$] \label{thm:main_result_sto} The following holds:
  \begin{enumerate}[label={\upshape(\roman*)}]
    \item The cost function $J(\delta)$ is the following quadratic function of the initial state
    \begin{equation}\label{eq:cost_function}
      J(\delta) = x_0^\top S_{0} x_0
    \end{equation}
    \item The gradient $\nabla J(\delta)$ of the cost function can be computed as
    \begin{equation}\label{eq:gradient}
      \nabla J(\delta)_i = \frac{\partial J(\delta)}{\partial \delta_i} = x_{i+1}^\top C_{i}x_{i+1}, \quad i = 0,\dots,N
    \end{equation}
    \item The Hessian $H_J(\delta)$ of the cost function can be computed as
    \ifTwoColumn
      \begin{equation}\label{eq:hessian}
        \begin{aligned}
        &H_J(\delta)_{i,\ell} = \frac{\partial^2 J(\delta)}{\partial \delta_i \partial \delta_{\ell}} \\
        &= \begin{dcases}2 x_{\ell+1}^\top C_{\ell} \Phi(\tau_{\ell+1}, \tau_{i+1})A_i^{n_i}x_{i+1}  & \ell \geq i\\
       H_J(\delta)_{\ell,i} & l < i,\end{dcases}
       \end{aligned}
      \end{equation}
    \else
      \begin{equation}\label{eq:hessian}
        H_J(\delta)_{i,\ell} = \frac{\partial^2 J(\delta)}{\partial \delta_i \partial \delta_{\ell}} = \begin{dcases}2 x_{\ell+1}^\top C_{\ell} \Phi(\tau_{\ell+1}, \tau_{i+1})A_i^{n_i}x_{i+1}  & \ell \geq i\\
       H_J(\delta)_{\ell,i} & l < i,\end{dcases}
      \end{equation}
      \fi
    where $i,\ell=0,\dots,N$.
\end{enumerate}
\end{theorem}
The proof can be found in Appendix~\ref{sec:Proof of Theorem thm:main_result_sto}.



Regardless of the second-order optimization method employed, most of the numerical operations needed to compute $J(\delta), \nabla J(\delta)$ and $H_J(\delta)$ at each iteration are shared. Thus, it is necessary to perform them only once per solver iteration.

\subsection{State Propagation and Matrix Exponentials}
\label{sub:State Propagation and Matrix Exponentials}

We now define the auxiliary matrices needed for our computations:
%
\begin{definition}[Auxiliary matrices]\label{def:auxiliary_matrices_sto}
  We define the matrix exponential of the linearized system between time instants $\tau_i^j$ and $\tau_{i}^{j+1}$ with $i=0,\dots,N$ and $j=0,\dots,n_i$ as
  \begin{equation}\label{eq:matcalE}
    \mathcal{E}_i^j \coloneqq e^{A_i^j\delta_{i}^j},
  \end{equation}
  Moreover, we define the matrices $M_i^j\in \mathbb{S}_{+}^{n_{x}}$ as
  \begin{equation}
    M_i^j \coloneqq \int_{0}^{\delta_i^j} e^{A_i^{j\top} \eta}  Q e^{A_i^j\eta} \mathrm{d}\eta.
  \end{equation}
\end{definition}

Both these matrices can be computed with the following single matrix exponential
\ifTwoColumn
\begin{equation}\label{eq:Zexp_l}
  Z_i^j = e^{G_i^j\delta_i^j} \coloneqq \begin{bmatrix}Z_{i,1}^j & Z_{i,2}^j\\ 0 & Z_{i,3}^j\end{bmatrix},
\end{equation}
with $Z_{i,1}^{j}, Z_{i,2}^{j}, Z_{i,3}^{j}\in \mathbb{R}^{n_x\times n_x}$
\else
  \begin{equation}\label{eq:Zexp_l}
    Z_i^j = e^{G_i^j\delta_i^j} \coloneqq \begin{bmatrix}Z_{i,1}^j & Z_{i,2}^j\\ 0 & Z_{i,3}^j\end{bmatrix}, \quad \text{with}\quad Z_{i,1}^{j}, Z_{i,2}^{j}, Z_{i,3}^{j}\in \mathbb{R}^{n_x\times n_x}.
  \end{equation}
\fi
and matrices $G_i^j$ being defined as
\begin{equation}
  G_i^j \coloneqq \begin{bmatrix}-A_i^{j\top} & Q\\ 0 & A_i^j\end{bmatrix}, \quad \text{with}\quad G_{i}^j \in \mathbb{R}^{2n_x \times 2n_x}.
\end{equation}
After computing $Z_i^j$, matrices $\mathcal{E}_i^j$ and  $M_i^j$ can be obtained as
\begin{equation}\label{eq:Eij, Mij}
  \mathcal{E}_i^j = Z_{i,3}^j \quad \mathrm{and \quad} M_i^j = Z_{i,3}^{j\top} Z_{i,2}^j.
\end{equation}
For more details, see~\cite[Theorem 1]{VanLoan:1978ff}.


In Algorithm~\ref{alg:linearize_compute_matrix_exponentials_propagate}
we describe the subroutine to propagate the state, linearize the dynamics and obtain matrices $\mathcal{E}_i^j$ and $M_i^j$.
\begin{algorithm}
\caption{Linearize, Compute Matrix Exponentials and Propagate}
\label{alg:linearize_compute_matrix_exponentials_propagate}
\begin{algorithmic}[1]
\Function{LinMatExpProp}{}
\For{$i=0,\dots,N$}
  \For{$j=0,\dots,n_i$}
  \State $A_{i}^j \gets$ Eq.~\eqref{eq:linearized_dynamics_sto1} \Comment{Linearize Dynamics}
  \State  $Z_i^j\gets $  Eq.~\eqref{eq:Zexp_l} \Comment{Matrix Exponential}
  \State $\mathcal{E}_i^j, M_i^j \gets$ Eq.~\eqref{eq:Eij, Mij}
  \State $x_{i}^{j+1} \gets \mathcal{E}_i^j x_i^j$
\EndFor
\EndFor
\State \Return $x_i,\quad i=0,\dots,N+1$
\State \Return $M_i^j, \mathcal{E}_i^j,\quad j=0,\dots,n_i\quad i=0,\dots,N$
\EndFunction
\end{algorithmic}
\end{algorithm}
At every instant $\tau_i^j$ the dynamics are linearized, matrices $\mathcal{E}_i^j,M_i^j$ are computed and the state  $x_{i}^j$ is propagated. Note that as we described in Section~\ref{sub:Time Grid and Dynamics Linearization}, we consider $x_i^0=x_i$ and $x_i^{n_i+1}=x_{i+1}$.

There are several methods to compute the matrix exponential as discussed in~\cite{Moler78nineteendubious} and~\cite{Moler:2003fn}. In our work we use the ``Method 3'' in~\cite[Section 3]{Moler:2003fn} being the scaling and squaring method explained in detail in~\cite{Higham:2009hh} which is in the main linear algebra library of the Julia language. The scaling and squaring method is the most common method used for computing the matrix exponential because of its efficiency and precision. However, in the case of linear dynamics discussed in Section~\ref{sec:Linear Switched Systems}, the matrices $A_i^j$ are always constant and many operations can be precomputed increasing the speed of the algorithm.

\subsubsection{Exponential Integrators}
\label{subs:Exponential Integrators}

The matrix exponentials employed in this section are an implementation of the first-order forward Euler exponential integrator~\cite{Hochbruck:2010cd}.  Exponential integrators perform well in many cases of stiff systems. However, most common numerical methods for exponential integration reduce the operations required  by computing directly the product of a matrix exponential and a vector. In our case, however, we not only need to propagate the dynamics, but also to compute the cost function integral. Thus, we need to compute the matrix exponentials $Z_i^j$ which are then used to compute $\mathcal{E}_i^j$ and $M_i^j$ from~\eqref{eq:Zexp_l}.

\subsection{State Transition Matrices $\Phi$}
\label{sub:State Transition Matrices}

From Theorem~\ref{thm:main_result_sto} we need the state transition matrices between the switching instants. They can be computed recursively using Definition~\ref{def:state_evolution} which, combined to the definition of the matrix exponentials in~\eqref{eq:matcalE}, can be written as
\begin{equation}\label{eq:state_tran_recursion}
  \Phi(\tau_l, \tau_i) = \prod_{q=i}^{l-1}\prod_{p=0}^{n_q}\mathcal{E}_q^{p}.
\end{equation}
Note that we need to compute the state transition matrices the case when $l\geq i$ so that the transition goes forward in time.

\subsection{Matrices $S_i^j$}
\label{sub:Matrices S}

To obtain the cost function and its first and second derivatives we need to compute matrices $S_{i}$. Given the matrices $\mathcal{E}_i^j$ and $M_i^j$, matrices $S_{i}$ can be obtained with the following proposition
\begin{proposition}\label{prop:S_recursion}
  Matrix $S_i^j$ with $i=0,\dots,N+1$ and $j=0,\dots,n_i$ satisfy the following recursion
  \begin{align}
    S_{N} &= E \label{eq:recursion_S1}\\
    S_{i}^j &= M_i^j + \mathcal{E}_i^{j\top} S_{i}^{j+1}\mathcal{E}_i^j \label{eq:recursion_S2}.
  \end{align}
\end{proposition}

The proof is in Appendix~\ref{sec:Proof of Proposition prop:S_recursion}.


Note that we are considering $S_i^0 = S_i$ and $S_i^{n_i+1} = S_{i+1}$ as discussed in Section~\ref{sub:Time Grid and Dynamics Linearization}.

\subsection{Complete Algorithm to Compute $J(\delta), \nabla J(\delta)$ and $H_J(\delta)$}
\label{sub:Complete Algorithm}

The complete algorithm to linearize problem~\eqref{eq:original_sto_problem} and compute the cost function, the gradient and the Hessian of~\eqref{eq:linearized_sto_problem} with respect to the switching intervals is shown in Algorithm~\ref{alg:complete_algorithm}.
\begin{algorithm}
\caption{Compute $J(\delta), \nabla J(\delta)$ and $H_J(\delta)$}
\label{alg:complete_algorithm}
\begin{algorithmic}[1]
\Function{ComputeCostFunctionAndDerivatives}{}
\Statex Shared Precomputations:
\State $x_i, \mathcal{E}_i^j, M_i^j \gets$ \Call{LinMatExpProp}{} \Comment{Algorithm~\ref{alg:linearize_compute_matrix_exponentials_propagate}}
\State $S_i \gets$ \Call{ComputeS}{} \Comment{Proposition~\ref{prop:S_recursion}}
\State $C_i \gets$ Eq.~\eqref{eq:C_matrix_sto} \Comment{Definition~\ref{def:C_matrix_sto}}
\State $\Phi(\tau_l,\tau_i) \gets$ \Call{Compute$\Phi$}{} \Comment{\eqref{eq:state_tran_recursion}}
\Statex Compute $J(\delta), \nabla J(\delta)$ and $H_J(\delta)$ \Comment{Theorem~\ref{thm:main_result_sto}}
\State $J(\delta)\gets$ Eq.~\eqref{eq:cost_function}
\State $\nabla J(\delta)\gets$ Eq.~\eqref{eq:gradient}
\State $H_J(\delta)\gets$ Eq.~\eqref{eq:hessian}
\EndFunction
\end{algorithmic}
\end{algorithm}
%

After performing the shared precomputations, the cost function and its derivatives can be computed using Theorem~\ref{thm:main_result_sto} with no significant increase in computation to obtain also the Hessian in order to apply a second-order method.

\section{Linear Switched Systems}
\label{sec:Linear Switched Systems}

When the system has linear switched dynamics of the form
\begin{equation}\label{eq:linear_switched_dynamics}
  \dot{x}(t) = A_i x(t),\quad t\in [\tau_i,\tau_{i+1}),\quad i=0,\dots,N
\end{equation}
the computations can be greatly simplified. In the main Algorithm~\ref{alg:solve_sto_problem} there is no need to resort to an auxiliary problem with linearized dynamics. In this case the main result in Theorem~\ref{thm:main_result_sto} applies directly to the cost function and derivatives of the original problem~\eqref{eq:original_sto_problem}.

There is no need for a linearization grid when dealing with linear systems. Thus, we simplify all the results for nonlinear dynamics by removing the indices $j$ by setting $n_i=0$ with $i=0,\dots,N+1$.

Since the dynamics matrices do not change during the optimization, we precompute the matrices in $G_i=G_i^0$ in~\eqref{eq:Zexp_l} offline. In addition, if some of the $G_i$ are diagonalizable, they can be factorized offline as
\begin{equation}
  G_i = Y_i^\top \Lambda_i Y_i, \quad i=0,\dots, N,
\end{equation}
where $\Lambda_i$ are the diagonal matrices of eigenvalues and $Y_i$ are the nonsingular matrices of right eigenvectors. Thus, matrix exponentials~\eqref{eq:Zexp_l} can be computed online as simple scalar exponentials of the diagonal elements of $\Lambda_i$
\begin{equation}
  Z_i = Y_i^\top e^{\Lambda_i\delta_i} Y_i, \quad i=0,\dots, N,
\end{equation}
which corresponds to ``Method 14'' in~\cite[Section 6]{Moler78nineteendubious} and~\cite{Moler:2003fn}. Note that the scalar exponentials are independent and can be computed in parallel to minimize the computation times. If $G_i$ are not diagonalizable, we compute the matrix exponentials as in the nonlinear system case with the scaling and squaring method~\cite[Section 3]{Moler:2003fn}.

Further improvements in computational efficiency can be obtained in the case of linear dynamics by executing the main for loop in Algorithm~\ref{alg:linearize_compute_matrix_exponentials_propagate} in parallel since there is no need to propagate the state and iteratively linearize the system.

The computational improvements when dealing with linear systems are shown in the examples section.

\section{Software and Examples}
\label{sec:Examples}

\label{sec:SwitchTimeOpt.jl}

All algorithms and examples described in this paper have been implemented in the open-source package \texttt{SwitchTimeOpt} in the Julia language, and are publicly available~\cite{SwitchTimeOpt}.
This package allows the user to easily define and efficiently solve switching time optimization problems for linear and nonlinear systems. \texttt{SwitchTimeOpt} supports a wide variety of nonlinear solvers through \texttt{MathProgBase} interface such as IPOPT~\cite{Wachter:2006hk} or KNITRO~\cite{Byrd:2006jv}.

For the complete documentation of the configurable options for defining problem~\eqref{eq:original_sto_problem} and the package functionalities we refer the reader to~\cite{SwitchTimeOpt}.

For each of the examples described in this section, we interfaced with the  \texttt{SwitchTimeOpt} IPOPT solver~\cite{Wachter:2006hk} on a late 2013 Macbook Pro with Intel Core i7 and 16GB of RAM.
All the examples are initialized with $\tau_i$ equally spaced between between $0$ and $T$. All the examples are solved with the default IPOPT options.

\subsection{Unstable Switched Dynamics}
\label{sub:Unstable Switched Dynamics}

Consider the switched system from~\cite{Caldwell:2012ei} described by the two unstable dynamics
\begin{equation}
  A_1 = \begin{bmatrix} -1 & 0\\ 1 & 2
\end{bmatrix}\quad \text{and}\quad A_2 = \begin{bmatrix}1 & 1\\ 1 & -2\end{bmatrix}.
\end{equation}
Note that $A_1$ and $A_2$ have no common eigenvectors. The system transitions happen $N=5$ times between $0$ and $T = 1$ according to the modes sequence ${\{ 1,2,1,2,1,2\}}$ and the cost function matrix is ${Q = I}$. The approach converges to precision $10^{-8}$ in roughly $\unit[3.5]{ms}$ producing the optimal switching times
\begin{equation*}
  \tau^* = \begin{bmatrix} 0.100 & 0.297 & 0.433 & 0.642 & 0.767\end{bmatrix}^\top
\end{equation*}
which correspond to the same solution obtained in~\cite{Caldwell:2012ei}. However, no timing is reported in that work.

To show the implementation ease of our software, we report in Listing~\ref{lst:listing_linear_example} the code of needed to produce this example:

\begin{lstlisting}[language=Julia, caption=\texttt{SwitchTimeOpt} code for the linear example., label=lst:listing_linear_example, float]
  N = 5  # Number of switching times

  # System dynamics
  A = zeros(nx, nx, N+1)
  A[:,:,1] = [-1 0; 1  2]
  A[:,:,2] = [ 1 1; 1 -2]
  for i = 3:N+1
    A[:,:,i] = A[:,:,mod(i+1,2)+1]
  end

  m = stoproblem(x0, A)  # Define problem
  solve!(m)              # Solve problem

  # Obtain results and timings
  tauopt = gettau(m)
  Jopt = getobjval(m)
  soltime = getsoltime(m)
\end{lstlisting}


\subsection{Lotka-Volterra Type Fishing Problem}
\label{sub:Lotka-Volterra Type Fishing Problem}

The Lotka-Volterra fishing problem has been studied for almost a century after D'Ancona and Volterra observed an unexpected decrease in fishing quota after World War I~\cite{Volterra:1926ws}. Lotka-Volterra systems present the nonlinear dynamics
\ifTwoColumn
\begin{equation}
   \dot{x}(t) \! = \!\begin{bmatrix}\dot{x}_1(t)\\ \dot{x}_2(t) \end{bmatrix} \!=\! \begin{bmatrix} x_1(t) - x_1(t)x_2(t) - c_1x_1(t)u(t)\\ -x_2(t) + x_1(t)x_2(t) - c_2x_2(t)u(t)\end{bmatrix}
\end{equation}
\else
\begin{equation}
   \dot{x}(t)  = \begin{bmatrix}\dot{x}_1(t)\\ \dot{x}_2(t) \end{bmatrix} = \begin{bmatrix} x_1(t) - x_1(t)x_2(t) - c_1x_1(t)u(t)\\ -x_2(t) + x_1(t)x_2(t) - c_2x_2(t)u(t)\end{bmatrix}
\end{equation}
\fi
defining the behavior of the biomass of the prey $x_1(t)$ assumed to grow exponentially and the predator $x_2(t)$ assumed to decrease exponentially. In addition, there is a coupling term describing the interaction of the biomasses when the predator eats the preys. The control action is the binary variable $u(t) \in \{ 0,1 \}$ consisting in the decision to fish $u(t)=1$ or not to fish $u(t)=0$ at time $t$. We choose $c_1=0.4$ and $c_2=0.2$ defining the number of preys and predators caught when fishing occurs.

This system has been analyzed from an integer optimal control point of view in~\cite{Sager:2006bt} and included in a library of standard integer optimal control benchmark problems for nonlinear systems in~\cite{Sager:2012ju}.

When no changes in the control action occur, i.e.\ we are either never fishing or always fishing, the system shows an oscillating behavior which can lead one of the biomasses to disappear~\cite{Sager:2006bt}, destroying the ecosystem. The goal is to responsibly fish in order to bring both the biomasses from an initial value of $x_0 = \begin{bmatrix}0.5 & 0.7\end{bmatrix}^\top$ to the steady state value $\begin{bmatrix}1 & 1\end{bmatrix}^\top$ within the time $T=12$. In other words, the optimal control problem consists in a tracking problem where we penalize the deviations from the reference  values $x_r(t) = \begin{bmatrix} 1 & 1\end{bmatrix}^\top$ by deciding when to start and stop fishing.

Given an integer input sequence $\{ u_i \}_{i=0}^{N},\; u_i \in \{ 0, 1\}$ and $N$ switching times $\tau_i$, the nonlinear dynamics can be described as a switched system of the form
\ifTwoColumn
\begin{equation}
  \dot{x}(t) \!=\! f_i(x(t)) \!=\! \begin{bmatrix} x_0(t) - x_0(t)x_1(t) - c_0x_0(t)u_i\\ -x_1(t) + x_0(t)x_1(t) - c_1x_1(t)u_i\end{bmatrix},
\end{equation}
with $t\in[\tau_{i}, \tau_{i+1}),\; i=0,\dots,N$.
\else
\begin{equation}
  \dot{x}(t) = f_i(x(t)) = \begin{bmatrix} x_0(t) - x_0(t)x_1(t) - c_0x_0(t)u_i\\ -x_1(t) + x_0(t)x_1(t) - c_1x_1(t)u_i\end{bmatrix}, \quad t\in[\tau_{i}, \tau_{i+1}),\; i=0,\dots,N.
\end{equation}
\fi
The complete optimal control problem can be written as
\ifTwoColumn
  \begin{equation}\label{eq:fishin_problem}
  \begin{aligned}
  &\underset{\delta}{\text{minimize}} &&\!\int_{0}^{T_\delta} \|x(t) - x_r(t)\|_2^2 \mathrm{d}t \\
  &\text{subject to} && \!\dot{x}(t) = f_i(x(t)),\; t\in[\tau_{i}, \tau_{i+1}),\; i=0,\dots,N\\
  &&& x(0)=x_0\\
  &&& \delta = \Delta.
  \end{aligned}
  \end{equation}
\else
  \begin{equation}\label{eq:fishin_problem}
  \begin{aligned}
  &\underset{\delta}{\text{minimize}} && \int_{0}^{T_\delta} \|x(t) - x_r(t)\|_2^2 \mathrm{d}t \\
  &\text{subject to} && \dot{x}(t) = f_i(x(t)),\quad t\in[\tau_{i}, \tau_{i+1}),\; i=0,\dots,N\\
  &&& x(0)=x_0\\
  &&& \delta = \Delta.
  \end{aligned}
  \end{equation}
\fi

The problem can be easily brought into the state-regulation form~\eqref{eq:original_sto_problem} by augmenting the state with $x_r(t)$ with $\dot{x}_r(t) = 0$ and minimizing the deviations between $x(t)$ and $x_r(t)$.

We consider a sequence of $N=8$ switchings between the two possible input values $\{u_i\}_{i=0}^{N} = \{0,1,0,1,0,1,0,1,0\}$ giving a total of $9$ dynamics.

We run the algorithm for $20$ iterations for increasing number of fixed-grid points $100, 150, 200$ and $250$. The optimal switching times for $n_{\mathrm{grid}} = 200$ are
\begin{equation}
  \tau^* = \begin{bmatrix} 2.446 & 4.150 & 4.533 & 4.799 & 5.436 & 5.616 & 6.969 & 7.033 \end{bmatrix}^\top,
\end{equation}
and the state behavior is displayed in Figure~\ref{fig:fish_problem}. The linearized system is also plotted as a dot-dashed green line showing an almost indistinguishable curve.

The complete results are shown in Table~\ref{tab:fishing_problem}. The system is simulated at the optimal intervals $\delta^*$ with an ode45 integrator obtaining the cost function value $J_{\mathrm{ode45}}(\delta^*)$ and with the grid linearizations obtaining $J(\delta^*)$ -- their values converge as the number of grid points increases. The latter can be seen from the value of $\Delta J = \|J_{\mathrm{ode45}}(\delta^*) - J(\delta^*)\|/\|J_{\mathrm{ode45}}(\delta^*)\|$ which decreases as the grid becomes finer. The number of cost function evaluations $n_{J, \mathrm{eval}}$ and the computation time are also shown in Table~\ref{tab:fishing_problem}. For the chosen solver IPOPT, increasing the number of grid points does not necessarily mean a higher computation time, because the latter is strictly related to the number of cost function evaluations which varies depending on the line search steps. We notice that, as the grid becomes finer, i.e.\ from $n_{\mathrm{grid}}=100$ to $150$, the linear approximation is more precise and the number of line search steps required is lower.

\begin{figure}
  \centering
  \ifTwoColumn
  \includegraphics[width=\columnwidth]{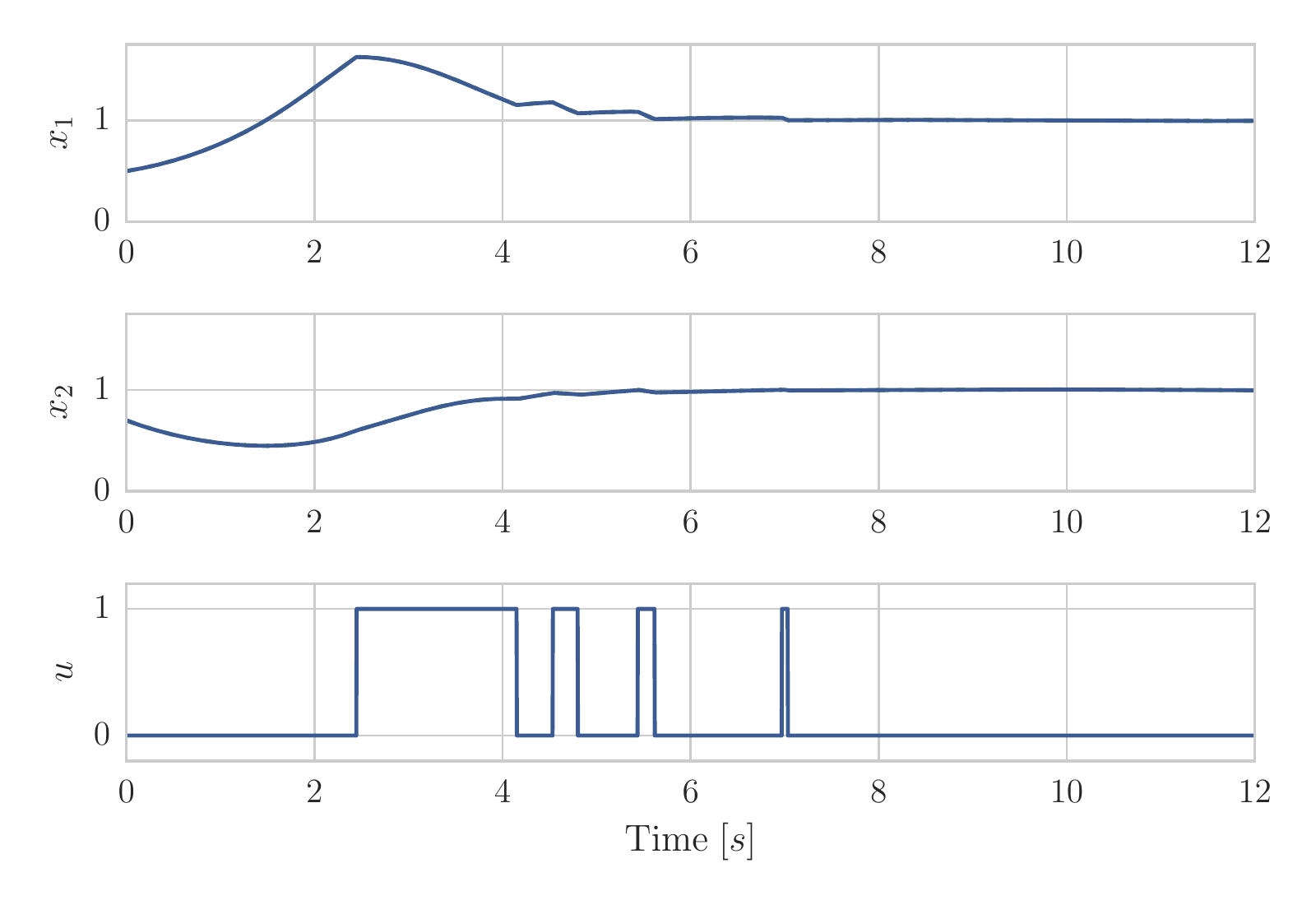}
  \else
  \includegraphics[width=0.8\columnwidth]{img/fishing_problem.pdf}
\fi
  \caption{Fishing problem. States and input behaviors at the optimal switching times~$\tau^*$. The states of the simulated nonlinear system (blue line) and the linearized system (dot-dashed green line) show a very close match.}
  \label{fig:fish_problem}
\end{figure}
\begin{table}
  \centering
    \caption{Results for Lotka-Volterra fishing problem after 20 iterations.}
    \label{tab:fishing_problem}
\begin{tabular}{rlllll}
\toprule
$n_{\mathrm{grid}}$ & $J_{\mathrm{ode45}}(\delta^*)$ & $J(\delta^*)$ & $\Delta J\; [\%]$ & $n_{J,\mathrm{eval}}$ & $\mathrm{Time}\; [s]$\\
 \midrule
 100 & 1.3500 & 1.3508 & 0.065 & 177 & 0.65\\
 150 & 1.3454 & 1.3459 & 0.033 & 56 & 0.27\\
 200 & 1.3456 & 1.3459 & 0.016 & 51 & 0.29\\
 250 & 1.3454 & 1.3455 & 0.010 & 54 & 0.38\\
     \bottomrule
  \end{tabular}
\end{table}



Our results are very close to the solutions in~\cite{Sager:2006bt} which are obtained with multiple shooting approach discretizing the problem a priori in $60$ time instants leading to a mixed-integer optimization problem with $2^{60}$ possible input combinations. In~\cite{Sager:2006bt} the authors deal with the required computational complexity by applying several heuristics. Their best cost function value is $1.3451$ and is obtained after solving the integer optimal control problem, applying a sum-up-rounding heuristic defined in~\cite{Sager:2006bt} as Rounding 2 and using the result to solve a switching time optimization problem with multiple shooting. Even though no timings are provided in~\cite{Sager:2006bt}, timing benchmarks for the multiple shooting approach applied to this problem are provided in the report~\cite[Section 5.5]{Sager:2011wu} where the execution times are approximately 10 times slower than the ones obtained in this work. Note that the implementations in~\cite{Sager:2006bt} and~\cite{Sager:2011wu} use the software package MUSCOD-II~\cite{Hoffmann:2011wy} which is a optimized C++ implementation of the multiple shooting methods, while our approach has been implemented on the high-level language Julia.

\subsection{Double-Tank System}
\label{sub:Double-Tank System}

The problem of controlling two interconnected tanks using hybrid control appeared in~\cite{Malmborg:1997tb}. The authors of~\cite{Axelsson:2005dd} applied switching time optimization to obtain the optimal inputs. This example has also been used in~\cite{Vasudevan:2013is} and~\cite{Claeys:2016cz} for relaxations in switched control systems.

The system dynamics can be written in the form
\begin{equation}\label{eq:double_tank}
  \dot{x}(t) = \begin{bmatrix}\dot{x}_1(t)\\\dot{x}_2(t)\end{bmatrix} = \begin{bmatrix}-\sqrt{x_1(t)} + u(t)\\
 \sqrt{x_1(t)} - \sqrt{x_2(t)} \end{bmatrix},
\end{equation}
where $x_1$ and $x_2$ are the fluid levels in the upper and lower tanks respectively. The control action $u(t)$ is the flow into the upper tank which is linked to the valve opening. We assume the input to be either $u_{\mathrm{min}}$ or $u_{\mathrm{max}}$. The goal of the control problem is tracking the reference level $x_r(t) = 3 -0.5t$ with the second tank (the tank is slowly emptying) over the time window from $0$ to $T$. The initial state is $x_0 = \begin{bmatrix}2 & 2\end{bmatrix}^\top$.

The optimal switching times problem has the same form as~\eqref{eq:fishin_problem}, but in this case the reference varies over time. We can bring the problem into state-regulation form by augmenting the state with $x_r(t)$ such that $\dot{x}_r(t) = -0.5$ and $x_r(0) = 3$ and minimizing deviations between $x_2(t)$ and $x_r(t)$.

We consider a sequence of $N=15$ switchings between $u_{\mathrm{min}}=2$ and $u_{\mathrm{max}}=3$ input values giving a total of $16$ dynamics.

We run the algorithm for $15$ iterations for increasing number of grid points $10,30,50$ and $100$. The optimal switching times for $n_{\mathrm{grid}} = 10$ are
\ifTwoColumn
  \begin{equation}
    \begin{multlined}[t][.85\columnwidth]\tau^* = [\begin{matrix}0.0 & 4.18 & 4.92 & 4.93 & 5.57 & 6.12 & 6.48 & 6.9 \end{matrix}\\
      \begin{matrix}7.26 & 7.67 & 8.04 & 8.43 & 8.81 & 9.19 & 9.56\end{matrix}]^\top.\end{multlined}
  \end{equation}
  \else
  \begin{equation}
    \tau^* = \begin{bmatrix}0.0 & 4.18 & 4.92 & 4.93 & 5.57 & 6.12 & 6.48 & 6.9 & 7.26 & 7.67 & 8.04 & 8.43 & 8.81 & 9.19 & 9.56\end{bmatrix}^\top.
  \end{equation}
  \fi

The behavior of the water levels is simulated with an ode45 integrator and displayed together with the valve opening in Figure~\ref{fig:double_tank}. The linearized system states behavior is plotted as a dot-dashed green line which coincides with the result from the nonlinear integrator. The dotted black line in the second plot represents the reference water level to be tracked.
\begin{figure}
  \centering
  \ifTwoColumn
  \includegraphics[width=\columnwidth]{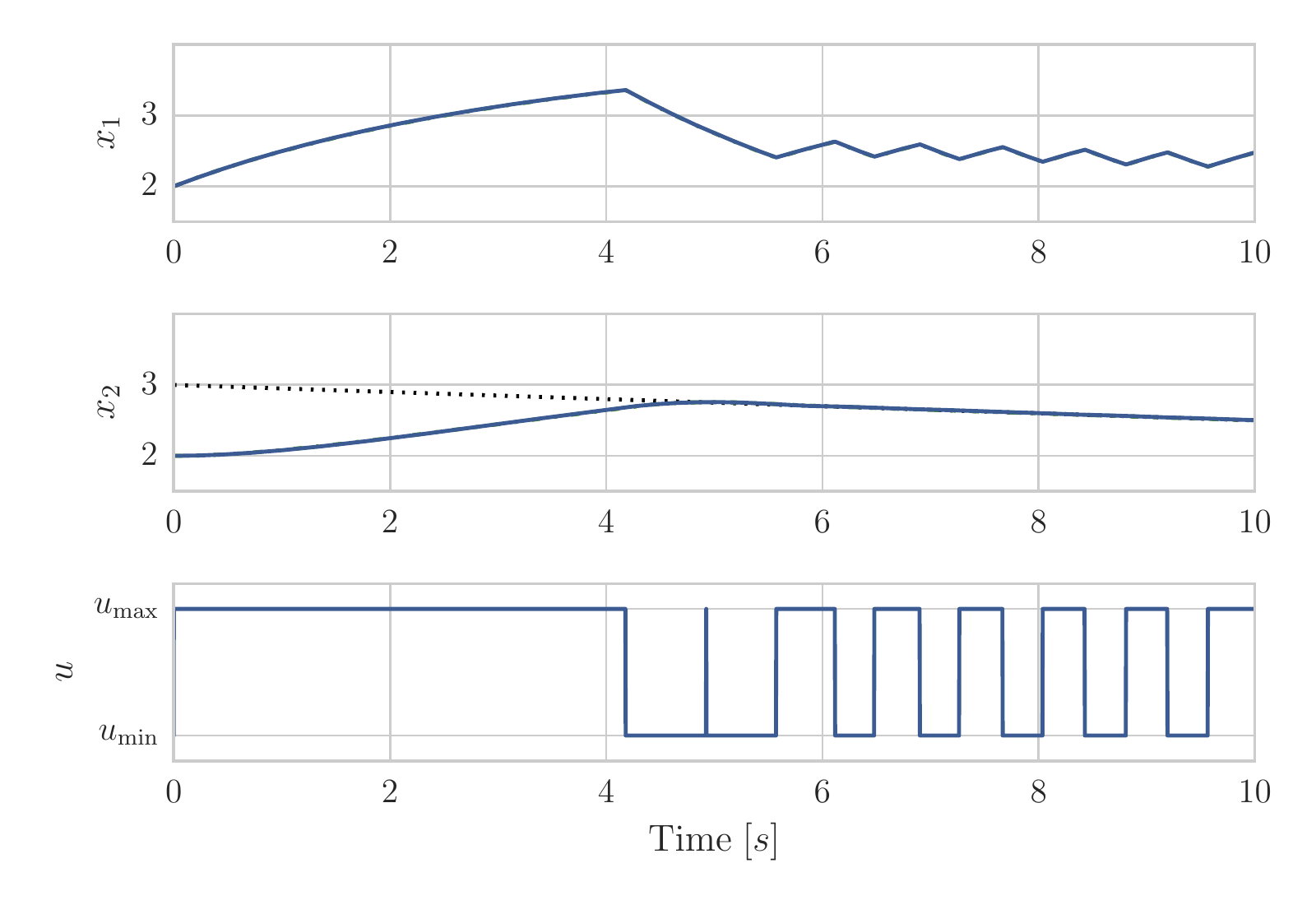}
  \else
  \includegraphics[width=0.8\columnwidth]{img/tank_problem.pdf}
\fi
  \caption{Double-tank system. States and input behaviors at the optimal switching times~$\tau^*$. The states of the simulated nonlinear system (blue line) and the linearized system (dot-dashed green line) show a very close match. The dotted black line in the second plot defines the reference to be tracked.}
  \label{fig:double_tank}
\end{figure}

The complete results are shown in Table~\ref{tab:tank_problem}. The nonlinear system is simulated at the optimal intervals $\delta^*$ obtaining the cost function $J_{\mathrm{ode45}}(\delta^*)$ and with the grid linearizations giving $J(\delta^*)$. As the shown in the table, the normalized absolute value of their difference tends to $0$ as the number of grid points increases. Even if the number of objective function evaluations is not monotonically increasing in the number of fixed grid points, we see an increasing execution time due to the required computations.
\begin{table}
  \centering
    \caption{Results for Double-tank problem after 15 iterations.}
    \label{tab:tank_problem}
\begin{tabular}{rlllll}
\toprule
$n_{\mathrm{grid}}$ & $J_{\mathrm{ode45}}(\delta^*)$ & $J(\delta^*)$ & $\Delta J\; [\%]$ & $n_{J,\mathrm{eval}}$ & $\mathrm{Time}\; [s]$\\
 \midrule
 10 & 1.8595 & 1.8495 & 0.537 & 39 & 0.05\\
 30 & 1.8582 & 1.8573 & 0.049 & 39 & 0.09\\
 50 & 1.8582 & 1.8578 & 0.021 & 49 & 0.11\\
 100 & 1.8582 & 1.8580 & 0.010 & 33 & 0.12\\
     \bottomrule
  \end{tabular}
\end{table}


Even though~\cite{Axelsson:2005dd} does not report computation times, in~\cite[Section 5.2]{Vasudevan:2013is} the authors report execution times in the order of $\unit[30]{sec}$ on an Intel Xeon, 12 core, 3.47 GHz, 92 GB RAM. Our approach is approximately $200$ to $550$ times faster on a standard laptop. Moreover, the problem described here is slightly more general since the reference is time-varying.

%
%
%
%

%

\section{Conclusion}
\label{sec:Conclusion}

We presented a novel method for computing the optimal switching times for linear and nonlinear switched systems.
By reformulating the problem with the switching intervals as optimization variables, we derive efficiently computable expressions for the cost function, the gradient and the Hessian which share the most expensive computations. Once the cost function value is obtained, at each iteration of the optimization algorithm, there is no significant increase in complexity in computing the gradient and the Hessian. In addition, we showed that in the case of linear dynamics many operations can be performed offline and many online operations parallelized greatly reducing the computation times.

We implemented our method in the open-source Julia package \texttt{SwitchTimeOpt} which allows the user to quickly define and solve optimal switching time problems. An example with linear dynamics shows that our method can solve switching time optimization problems in milliseconds time scale. We also show with two nonlinear dynamics examples that our high-level Julia implementation can solve these problems with one up to two orders of magnitude improvements over state-of-the-art approaches.

There are several future directions to be investigated. First of all,  many computations can be parallelized. The state transition matrices in~\eqref{eq:state_tran_recursion} can be computed in parallel for every different $\tau_i$. Moreover, given the associative nature of the matrix products, parallel reduction techniques like the prefix-sum~\cite{ladner1980parallel} could be implemented to reduce the computation time. In the case of linear dynamics, since there is no need to sequentially propagate the state before the linearizations, the matrix exponentials can be computed completely in parallel.  Note that Julia language already includes several functions to parallelize computations on standard CPUs. However, we believe that instead our approach could greatly benefit from implementations of these parallelizations on CUDA or FPGA architectures.
Another research direction could be to develop a tailored solver to our method to exploit its structure. For example, interior point methods such as IPOPT, exploit line search routines which could end up evaluating the cost function several times increasing the computation time due to the matrix exponentials computations at each different point. An optimization algorithm taking into account the most expensive computations in our subroutines, could definitely increase the performance. Finally, the current work could be extended to more general problem formulations such as optimal control with state constraints.

  \ifTechReport \appendix 
  \else \appendices \fi  


\section{Proof of Theorem~\ref{thm:main_result_sto}}
\label{sec:Proof of Theorem thm:main_result_sto}

In order to prove the main result, we first require the following two lemmas:

\begin{lemma}[State transition matrix derivative] \label{lem:state_trans_mat_deriv} Given two switching times $\tau_a$, $\tau_{i+1}$  with $\tau_a\leq \tau_{i+1}$ such  that $\tau_{i+1}$ does not coincide with any point of the background grid and the switching interval $\delta_i$, the first derivative of the state transition matrix between $\tau_a$ and $\tau_{i+1}$ with respect to $\delta_i$ can be written as
  \begin{equation}
    \frac{\partial \Phi(\tau_{i+1},\tau_a)}{\partial \delta_i} = A_{i}^{n_{i}}\Phi(\tau_{i+1},\tau_a).
  \end{equation}
\end{lemma}

Note that in the case when $\tau_{i+1}$ coincides with a fixed-grid point, the derivative is not defined since at $\tau_{i+1} + \epsilon$ a new linearization is introduced breaking the smoothness of the state transition matrix. Our derivations still hold in that case by considering, instead of the gradient, the subgradient equal to the one-sided limit of the derivative from below.

\begin{proof}
  We can rewrite $\Phi(\tau_{i+1},\tau_a)$ using Definition~\ref{def:state_evolution} as
  \begin{equation}\label{eq:state_trans_deriv_proof1}
    \Phi(\tau_{i+1}, \tau_a) =  e^{A_i^{n_i}\left(\delta_{i} - \sum_{p = 0}^{n_{i}-1}\delta_{i}^{p}\right)} \left(\prod_{p=0}^{n_{i}-1}e^{A_{i}^p\delta_{i}^p}\right) \Phi(\tau_{i},\tau_a),
  \end{equation}
  by using the relation
  \begin{equation}
    \delta_{i} = \sum_{j = 0}^{n_i}\delta_{i}^{j}.
  \end{equation}
  Taking the derivative of~\eqref{eq:state_trans_deriv_proof1} we obtain:
  \ifTwoColumn
  \begin{equation}
    \begin{aligned}
    &\frac{\partial \Phi(\tau_{i+1},\tau_a)}{\partial \delta_i} =\\ &= A_i^{n_i}e^{A_i^{n_i}\left(\delta_{i} - \sum_{p = 0}^{n_{i}-1}\delta_{i}^{p}\right)} \left(\prod_{p=0}^{n_{i}-1}e^{A_{i}^p\delta_{i}^p}\right) \Phi(\tau_{i},\tau_a)\\
      &= A_i^{n_i} \Phi(\tau_{i+1},\tau_a),
    \end{aligned}
  \end{equation}
  \else
  \begin{equation}
    \begin{aligned}
      \frac{\partial \Phi(\tau_{i+1},\tau_a)}{\partial \delta_i} &= A_i^{n_i}e^{A_i^{n_i}\left(\delta_{i} - \sum_{p = 0}^{n_{i}-1}\delta_{i}^{p}\right)} \left(\prod_{p=0}^{n_{i}-1}e^{A_{i}^p\delta_{i}^p}\right) \Phi(\tau_{i},\tau_a)\\
      &= A_i^{n_i} \Phi(\tau_{i+1},\tau_a),
    \end{aligned}
  \end{equation}
  \fi
  where we made use of the properties of the matrix exponential $e^{X(a+b)} = e^{Xa} + e^{Xb}$ and $\frac{\partial}{\partial c} e^{Xc} = X e^{Xc}$ with $X\in\mathbb{R}^{n_x}$ and $a,b,t \in \mathbb{R}$.
\end{proof}

The matrices $S_i$ and their first derivatives play an important role in the proof and in the rest of the paper. We here derive the first derivative of $S_i$.

\begin{lemma}[Derivative of Matrices $S_i$]\label{lem:deriv_S}
  Given the switching times $\tau_a$ and $\tau_{i+1}$ so that $\tau_a \leq \tau_{i+1}$ and that $\tau_{i+1}$ does not coincide with any point of the background grid and the interval $\delta_i$, the derivative of $S_a$ with respect to $\delta_i$ is
  \begin{equation}
    \frac{\partial S_{a}}{\partial \delta_i} = \Phi(\tau_{i+1},\tau_a)^\top C_{i} \Phi(\tau_{i+1},\tau_a)
  \end{equation}
\end{lemma}


\begin{proof}
  From~\eqref{eq:S_sum_P_F}, we can write the derivative as
  \begin{equation}\label{eq:deriv_S_sum_P_F}
    \frac{\partial S_{a}}{\partial \delta_i} = \frac{\partial P_{a}}{\partial \delta_i} + \frac{\partial F_{a}}{\partial \delta_i}.
  \end{equation}
  Let us analyze the two components separately. We decompose the integral defined by $P_{a}$ as
  \ifTwoColumn
  \begin{align}
    \frac{\partial P_{a}}{\partial \delta_i} & = \frac{\partial}{\partial \delta_i}\left(\int_{\tau_a}^{T_\delta} \Phi(t,\tau_a)^\top Q \Phi(t,\tau_a)\mathrm{d}t\right) \nonumber\\
    &= \begin{multlined}[t][0.85\columnwidth]\frac{\partial}{\partial \delta_i}\left(\int_{\tau_a}^{\tau_{i}} \Phi(t,\tau_a)^\top Q \Phi(t,\tau_a)\mathrm{d}t\right) \\ + \frac{\partial}{\partial \delta_i}\left(\int_{\tau_{i}}^{\tau_{i+1}} \Phi(t,\tau_a)^\top Q \Phi(t,\tau_a)\mathrm{d}t\right) \\+ \frac{\partial}{\partial \delta_i}\left(\int_{\tau_{i+1}}^{T_\delta} \Phi(t,\tau_a)^\top Q \Phi(t,\tau_a)\mathrm{d}t\right)
  \end{multlined}\nonumber\\
  &= \begin{multlined}[t][0.85\columnwidth]\frac{\partial}{\partial \delta_i}\left(\int_{\tau_{i}}^{\tau_{i+1}} \Phi(t,\tau_a)^\top Q \Phi(t,\tau_a)\mathrm{d}t\right) \\+ \frac{\partial}{\partial \delta_i}\left(\int_{\tau_{i+1}}^{T_\delta} \Phi(t,\tau_a)^\top Q \Phi(t,\tau_a)\mathrm{d}t\right).\end{multlined}\label{eq:P_taui_sum}
  \end{align}
  \else
  \begin{align}
    \frac{\partial P_{a}}{\partial \delta_i} & = \frac{\partial}{\partial \delta_i}\left(\int_{\tau_a}^{T_\delta} \Phi(t,\tau_a)^\top Q \Phi(t,\tau_a)\mathrm{d}t\right) \nonumber\\
    &= \begin{multlined}[t]\frac{\partial}{\partial \delta_i}\left(\int_{\tau_a}^{\tau_{i}} \Phi(t,\tau_a)^\top Q \Phi(t,\tau_a)\mathrm{d}t\right) + \\ \frac{\partial}{\partial \delta_i}\left(\int_{\tau_{i}}^{\tau_{i+1}} \Phi(t,\tau_a)^\top Q \Phi(t,\tau_a)\mathrm{d}t\right) + \frac{\partial}{\partial \delta_i}\left(\int_{\tau_{i+1}}^{T_\delta} \Phi(t,\tau_a)^\top Q \Phi(t,\tau_a)\mathrm{d}t\right)
  \end{multlined}\nonumber\\
  &= \frac{\partial}{\partial \delta_i}\left(\int_{\tau_{i}}^{\tau_{i+1}} \Phi(t,\tau_a)^\top Q \Phi(t,\tau_a)\mathrm{d}t\right) + \frac{\partial}{\partial \delta_i}\left(\int_{\tau_{i+1}}^{T_\delta} \Phi(t,\tau_a)^\top Q \Phi(t,\tau_a)\mathrm{d}t\right)\label{eq:P_taui_sum}.
  \end{align}
  \fi
  Note that the integral from $\tau_a$ to $\tau_{i}$  does not depend on $\delta_i$ and its derivative is zero. Taking first the leftmost term in~\eqref{eq:P_taui_sum}, the integral from $\tau_{i}$ to $\tau_{i+1}$ can be written as
  \ifTwoColumn
  \begin{align}
    &\frac{\partial}{\partial \delta_i}\left(\int_{\tau_a}^{\tau_{i+1}} \Phi(t,\tau_a)^\top Q \Phi(t,\tau_a)\mathrm{d}t\right)\\ &=\Phi(\tau_{i},\tau_a)^\top \frac{\partial}{\partial \delta_i}\left(\int_{\tau_{i}}^{\tau_{i+1}} \Phi(t,\tau_{i})^\top Q \Phi(t,\tau_{i})\mathrm{d}t\right) \Phi(\tau_{i},\tau_a)\\
    &= \begin{multlined}[t][0.85\columnwidth]\Phi(\tau_{i},\tau_a)^\top\frac{\partial}{\partial \delta_i}\bigg(\int_{0}^{\delta{i}} \Phi(\eta + \tau_{i},\tau_{i})^\top Q \\ \cdot \Phi(\eta+\tau_{i},\tau_{i})\mathrm{d}\eta\bigg) \Phi(\tau_{i},\tau_{a})\end{multlined}\nonumber\\
    &= \Phi(\tau_{i+1},\tau_{a})^\top Q \Phi(\tau_{i+1},\tau_{a}),\label{eq:P_taui_sum1}
  \end{align}
  \else
  \begin{align}
    &\frac{\partial}{\partial \delta_i}\left(\int_{\tau_a}^{\tau_{i+1}} \Phi(t,\tau_a)^\top Q \Phi(t,\tau_a)\mathrm{d}t\right)\\ &=\Phi(\tau_{i},\tau_a)^\top \frac{\partial}{\partial \delta_i}\left(\int_{\tau_{i}}^{\tau_{i+1}} \Phi(t,\tau_{i})^\top Q \Phi(t,\tau_{i})\mathrm{d}t\right) \Phi(\tau_{i},\tau_a)\\
    &= \Phi(\tau_{i},\tau_a)^\top\frac{\partial}{\partial \delta_i}\left(\int_{0}^{\delta{i}} \Phi(\eta + \tau_{i},\tau_{i})^\top Q \Phi(\eta+\tau_{i},\tau_{i})\mathrm{d}\eta\right) \Phi(\tau_{i},\tau_{a})\nonumber\\
    &= \Phi(\tau_{i+1},\tau_{a})^\top Q \Phi(\tau_{i+1},\tau_{a}),\label{eq:P_taui_sum1}
  \end{align}
  \fi
  where in the second equality we applied the change of variables $\eta = t - \tau_i$ and in the third equality the fundamental theorem of calculus. Next taking the rightmost term in~\eqref{eq:P_taui_sum}, the integral from $\tau_{i+1}$ to $T_\delta$ can be obtained as
  \ifTwoColumn
  \begin{align}
    &\frac{\partial}{\partial \delta_i}\left(\int_{\tau_{i+1}}^{T_\delta} \Phi(t,\tau_{a})^\top Q \Phi(t,\tau_{a})\mathrm{d}t\right) = \nonumber\\
    &=\begin{multlined}[t][0.9\columnwidth]\frac{\partial}{\partial \delta_i}\bigg(   \Phi(\tau_{i+1}, \tau_{a})^\top\bigg(\int_{\tau_{i+1}}^{T_\delta} \Phi(t,\tau_{i+1})^\top Q \\ \cdot \Phi(t,\tau_{i+1})\mathrm{d}t\bigg) \Phi(\tau_{i+1}, \tau_{a})\bigg)\end{multlined}\nonumber\\
    &=\frac{\partial}{\partial \delta_i}\left(   \Phi(\tau_{i+1}, \tau_{a})^\top P_{i+1} \Phi(\tau_{i+1}, \tau_{a})\right)\nonumber\\
    &= \begin{multlined}[t][0.9\columnwidth]\frac{\partial}{\partial \delta_i}\bigg(   \Phi(\tau_{i+1}, \tau_{a})^\top\bigg) P_{i+1} \Phi(\tau_{i+1}, \tau_{a}) \\ +    \Phi(\tau_{i+1}, \tau_{a})^\top P_{i+1} \frac{\partial}{\partial \delta_i}\bigg(\Phi(\tau_{i+1}, \tau_{a})\bigg)\end{multlined}\nonumber\\
    &=  \Phi(\tau_{i+1}, \tau_{a})^\top \left((A_{i}^{n_i})^\top  P_{i+1} +  P_{i+1}  A_{i}^{n_i}\right)  \Phi(\tau_{i+1}, \tau_{a}). \label{eq:P_taui_sum2}
  \end{align}
  \else
  \begin{align}
    &\frac{\partial}{\partial \delta_i}\left(\int_{\tau_{i+1}}^{T_\delta} \Phi(t,\tau_{a})^\top Q \Phi(t,\tau_{a})\mathrm{d}t\right) = \nonumber\\
    &=\frac{\partial}{\partial \delta_i}\left(   \Phi(\tau_{i+1}, \tau_{a})^\top\left(\int_{\tau_{i+1}}^{T_\delta} \Phi(t,\tau_{i+1})^\top Q \Phi(t,\tau_{i+1})\mathrm{d}t\right) \Phi(\tau_{i+1}, \tau_{a})\right)\nonumber\\
    &=\frac{\partial}{\partial \delta_i}\left(   \Phi(\tau_{i+1}, \tau_{a})^\top P_{i+1} \Phi(\tau_{i+1}, \tau_{a})\right)\nonumber\\
    &= \frac{\partial}{\partial \delta_i}\bigg(   \Phi(\tau_{i+1}, \tau_{a})^\top\bigg) P_{i+1} \Phi(\tau_{i+1}, \tau_{a}) +    \Phi(\tau_{i+1}, \tau_{a})^\top P_{i+1} \frac{\partial}{\partial \delta_i}\bigg(\Phi(\tau_{i+1}, \tau_{a})\bigg)\nonumber\\
    &=  \Phi(\tau_{i+1}, \tau_{a})^\top \left((A_{i}^{n_i})^\top  P_{i+1} +  P_{i+1}  A_{i}^{n_i}\right)  \Phi(\tau_{i+1}, \tau_{a}). \label{eq:P_taui_sum2}
  \end{align}
  \fi
  In the first and second equalities we decomposed the state transition matrices and used the definition of $P_{i+1}$ of~\eqref{eq:def_P}. In the third equality we applied the chain rule noting that $P_{i+1}$ is independent from $\delta_i$. Then, in the last equality we applied Lemma~\ref{lem:state_trans_mat_deriv} to compute the derivatives. We rewrite Equation~\eqref{eq:P_taui_sum} using~\eqref{eq:P_taui_sum1} and~\eqref{eq:P_taui_sum2} obtaining
  \ifTwoColumn
  \begin{equation}\label{eq:deriv_P_final}
  \frac{\partial P_{a}}{\partial \delta_i} =\Phi(\tau_{i+1}, \tau_{a})^\top \!\bigg(Q + (A_{i}^{n_i})^\top  P_{i+1} + P_{i+1}  A_{i}^{n_i}\bigg)\! \Phi(\tau_{i+1}, \tau_{a}).
  \end{equation}
  \else
  \begin{equation}\label{eq:deriv_P_final}
    \frac{\partial P_{a}}{\partial \delta_i} = \Phi(\tau_{i+1}, \tau_{a})^\top \left(Q + (A_{i}^{n_i})^\top  P_{i+1} +  P_{i+1}  A_{i}^{n_i}\right)  \Phi(\tau_{i+1}, \tau_{a}).
  \end{equation}
  \fi

  We now focus on the derivative of $F_{a}$ in~\eqref{eq:deriv_S_sum_P_F} that can be written so that
  \ifTwoColumn
  \begin{align}
    &\frac{\partial F_{a}}{\partial \delta_i} = \frac{\partial}{\partial \delta_i}\left(\Phi(T_\delta, \tau_{a})^\top E  \Phi(T_\delta, \tau_{a})\right) \nonumber\\
    &= \frac{\partial}{\partial \delta_i}\bigg(\!\Phi(\tau_{i+1}, \tau_{a})^\top \!\bigg(\!\Phi(T_\delta, \tau_{i+1})^\top \!E  \Phi(T_\delta, \tau_{i+1})\!\bigg)\Phi(\tau_{i+1}, \tau_{a})\!\bigg) \nonumber\\
    &= \frac{\partial}{\partial \delta_i}\bigg(\Phi(\tau_{i+1}, \tau_{a})^\top F_{i+1}\Phi(\tau_{i+1}, \tau_{a})\bigg)  \nonumber\\
    &= \begin{multlined}[t][0.9\columnwidth]\frac{\partial}{\partial \delta_i}\bigg(\Phi(\tau_{i+1}, \tau_{a})^\top\bigg) F_{i+1}\Phi(\tau_{i+1}, \tau_{a})  \\+ \Phi(\tau_{i+1}, \tau_{a})^\top F_{i+1}\frac{\partial}{\partial \delta_i}\bigg(\Phi(\tau_{i+1}, \tau_{a})\bigg) \end{multlined}\nonumber\\
    &= \Phi(\tau_{i+1}, \tau_{a})^\top \bigg((A_{i}^{n_i})^\top  F_{i+1} +  F_{i+1}  A_{i}^{n_i}\bigg)\Phi(\tau_{i+1}, \tau_{a}),\label{eq:deriv_F_final}
  \end{align}
  \else
  \begin{align}
    &\frac{\partial F_{a}}{\partial \delta_i} = \frac{\partial}{\partial \delta_i}\left(\Phi(T_\delta, \tau_{a})^\top E  \Phi(T_\delta, \tau_{a})\right) \nonumber\\
    &= \frac{\partial}{\partial \delta_i}\bigg(\Phi(\tau_{i+1}, \tau_{a})^\top \bigg(\Phi(T_\delta, \tau_{i+1})^\top E  \Phi(T_\delta, \tau_{i+1})\bigg)\Phi(\tau_{i+1}, \tau_{a})\bigg) \nonumber\\
    &= \frac{\partial}{\partial \delta_i}\bigg(\Phi(\tau_{i+1}, \tau_{a})^\top F_{i+1}\Phi(\tau_{i+1}, \tau_{a})\bigg)  \nonumber\\
    &= \frac{\partial}{\partial \delta_i}\bigg(\Phi(\tau_{i+1}, \tau_{a})^\top\bigg) F_{i+1}\Phi(\tau_{i+1}, \tau_{a})  + \Phi(\tau_{i+1}, \tau_{a})^\top F_{i+1}\frac{\partial}{\partial \delta_i}\bigg(\Phi(\tau_{i+1}, \tau_{a})\bigg) \nonumber\\
    &= \Phi(\tau_{i+1}, \tau_{a})^\top \bigg((A_{i}^{n_i})^\top  F_{i+1} +  F_{i+1}  A_{i}^{n_i}\bigg)\Phi(\tau_{i+1}, \tau_{a}),\label{eq:deriv_F_final}
  \end{align}
  \fi
  In the first and second equalities we decomposed the state transition matrices and used the definition of $F_{i+1}$ from~\eqref{eq:def_F}. In the third equality we applied the chain rule noting that $F_{i+1}$ is independent from $\delta_i$. Then, in the last equality we applied Lemma~\ref{lem:state_trans_mat_deriv} to compute the derivatives.

  By adding~\eqref{eq:deriv_P_final} and~\eqref{eq:deriv_F_final} as in~\eqref{eq:deriv_S_sum_P_F} and applying Definition~\ref{def:matrices_P_F_S} we obtain
  \begin{equation}
    \frac{\partial S_{a}}{\partial \delta_i} = \Phi(\tau_{i+1},\tau_{a})^\top \bigg(Q + \left(A_i^{n_i}\right)^\top S_{i+1} + S_{i+1}A_i^{n_i} \bigg) \Phi(\tau_{i+1},\tau_{a}).
  \end{equation}
  The result follows by using Definition~\ref{def:C_matrix_sto}.
\end{proof}

We are now in a position to prove each of the statements in Theorem~\ref{thm:main_result_sto} in turn:

\paragraph{Cost function -- proof of (i)}
\label{par:Cost Function}
The cost function in Equation~\eqref{eq:cost_function} can be directly derived from its definition in Problem~\eqref{eq:linearized_sto_problem} and Definition~\ref{def:matrices_P_F_S}.

\paragraph{Gradient -- proof of (ii)}
The gradient of the cost function can be derived by taking the derivative of~\eqref{eq:cost_function}. By considering the component related to $\delta_i$, we can write
\begin{align}
  \frac{\partial J(\delta)}{\partial \delta_i} &= \frac{\partial}{\partial \delta_i}\bigg( x_0^\top S_{0} x_0\bigg) \nonumber\\
  &=  x_0^\top\frac{\partial S_{0}}{\partial \delta_i}  x_0 \nonumber\\
  &= x_0^\top \Phi(\tau_{i+1},0)^\top C_{i} \Phi(\tau_{i+1},0) x_0 \nonumber\\
  &=x_{i+1}^\top C_{i}x_{i+1}.\label{eq:gradient_sto_proof}
\end{align}
In the second equality the initial state has been taken out from the derivative operator since $x_0$ fixed.
In the third equality we applied Lemma~\ref{lem:deriv_S} and in the fourth equality we used Definition~\ref{def:state_evolution} to obtain $x_{i+1}$.
The result holds for $i=0,\dots,N$.

\paragraph{Hessian -- proof of (iii)}

The Hessian of the cost function can be derived by taking the derivative of Equation~\eqref{eq:gradient_sto_proof}.
Let us first take the derivative with respect to the same interval $\delta_i$ writing
\ifTwoColumn
\begin{align}
  &\frac{\partial^2 J(\delta)}{\partial \delta_i^2 } = \frac{\partial}{\partial \delta_i}\bigg( x_{i+1}^\top C_{i}x_{i+1}\bigg) \nonumber\\
  &= \frac{\partial}{\partial \delta_i}\bigg( x_0^\top \Phi(\tau_{i+1},0)^\top  C_{i} \Phi(\tau_{i+1},0)x_0\bigg)\nonumber\\
  &= \begin{multlined}[t][0.85\columnwidth]x_0^\top \frac{\partial}{\partial \delta_i}\bigg( \Phi(\tau_{i+1},0)^\top \bigg) C_{i} \Phi(\tau_{i+1},0)x_0 \\+  x_0^\top \Phi(\tau_{i+1},0)^\top  C_{i}\frac{\partial}{\partial \delta_i}\bigg(  \Phi(\tau_{i+1},0)\bigg)x_0 \end{multlined}\nonumber\\
  &= \begin{multlined}[t][0.85\columnwidth]x_0^\top \Phi(\tau_{i+1},0)^\top \left(A_{i}^{n_i}\right)^\top C_{i} \Phi(\tau_{i+1},0)x_0 \\+  x_0^\top \Phi(\tau_{i+1},0)^\top  C_{i}  A_{i}^{n_i} \Phi(\tau_{i+1},0) x_0\end{multlined}\nonumber\\
  &= x_{i+1}^\top \left(A_{i}^{n_i}\right)^\top C_{i} x_{i+1} +  x_{i+1}^\top  C_{i}  A_{i}^{n_i} x_{i+1}\nonumber\\
  &= 2x_{i+1}^\top  C_{i}  A_{i}^{n_i} x_{i+1}\nonumber\\
  &= 2x_{i+1}^\top  C_{i}  \Phi(\tau_{i+1},\tau_{i+1})A_{i}^{n_i} x_{i+1}.
\end{align}
\else
\begin{align}
  &\frac{\partial^2 J(\delta)}{\partial \delta_i^2 } = \frac{\partial}{\partial \delta_i}\bigg( x_{i+1}^\top C_{i}x_{i+1}\bigg) \nonumber\\
  &= \frac{\partial}{\partial \delta_i}\bigg( x_0^\top \Phi(\tau_{i+1},0)^\top  C_{i} \Phi(\tau_{i+1},0)x_0\bigg)\nonumber\\
  &= x_0^\top \frac{\partial}{\partial \delta_i}\bigg( \Phi(\tau_{i+1},0)^\top \bigg) C_{i} \Phi(\tau_{i+1},0)x_0 +  x_0^\top \Phi(\tau_{i+1},0)^\top  C_{i}\frac{\partial}{\partial \delta_i}\bigg(  \Phi(\tau_{i+1},0)\bigg)x_0 \nonumber\\
  &= x_0^\top \Phi(\tau_{i+1},0)^\top \left(A_{i}^{n_i}\right)^\top C_{i} \Phi(\tau_{i+1},0)x_0 +  x_0^\top \Phi(\tau_{i+1},0)^\top  C_{i}  A_{i}^{n_i} \Phi(\tau_{i+1},0) x_0\nonumber\\
  &= x_{i+1}^\top \left(A_{i}^{n_i}\right)^\top C_{i} x_{i+1} +  x_{i+1}^\top  C_{i}  A_{i}^{n_i} x_{i+1}\nonumber\\
  &= 2x_{i+1}^\top  C_{i}  A_{i}^{n_i} x_{i+1}\nonumber\\
  &= 2x_{i+1}^\top  C_{i}  \Phi(\tau_{i+1},\tau_{i+1})A_{i}^{n_i} x_{i+1}.
\end{align}
\fi
In the third equality we took into account that $C_{i}$ and $x_0$ do not depend on $\delta_i$ and we applied the chain rule. In the fourth equality we applied Lemma~\ref{lem:state_trans_mat_deriv} and in the fifth equality Definition~\ref{def:state_evolution}. Finally, in the sixth equality we took the transpose of the first term which is a scalar and we used the identity $I = \Phi(\tau_{i+1}, \tau_{i+1})$ to get the desired result ($I$ is the identity matrix of appropriate dimensions).

In the case when we take the derivative with respect to $\delta_\ell$ with $\ell>i$, we can write
\begin{align}
    &\frac{\partial^2 J(\delta)}{\partial \delta_\ell \partial \delta_i } = \frac{\partial}{\partial \delta_\ell}\bigg( x_{i+1}^\top C_{i}x_{i+1}\bigg) \nonumber\\
    &=\frac{\partial}{\partial \delta_\ell}\bigg( x_{i+1}^\top \bigg(Q +   A_i^{n_i\top} S_{i+1} + S_{i+1}A_i^{n_i}\bigg)x_{i+1}\bigg) \nonumber\\
    &=x_{i+1}^\top A_i^{n_i\top} \frac{\partial}{\partial \delta_\ell}\bigg( S_{i+1}\bigg)x_{i+1} + x_{i+1}^\top \frac{\partial}{\partial \delta_\ell}\bigg(S_{i+1}\bigg)A_i^{n_i}x_{i+1} \nonumber\\
    &=2x_{i+1}^\top \frac{\partial}{\partial \delta_\ell}\bigg(S_{i+1}\bigg)A_i^{n_i}x_{i+1} \nonumber\\
    &=2x_{i+1}^\top \Phi(\tau_{\ell+1},\tau_{i+1})^\top C_{\ell} \Phi(\tau_{\ell+1},\tau_{i+1})A_i^{n_i}x_{i+1} \nonumber\\
    &=2x_{\ell+1}^\top C_{\ell} \Phi(\tau_{\ell+1},\tau_{i+1})A_i^{n_i}x_{i+1}
\end{align}
in the second equality we applied Definition~\ref{def:C_matrix_sto}. In the third inequality we brought the terms not depending on $\delta_{\ell}$ outside of the derivative operator. In the fourth equality we took the transpose of the first element which is a scalar. In the fifth equality we applied Lemma~\ref{lem:deriv_S} and finally we obtained $x_{\ell+1}$ from Definition~\ref{def:state_evolution}.

\hfill\qedsymbol

\section{Proof of Proposition~\ref{prop:S_recursion}}
\label{sec:Proof of Proposition prop:S_recursion}

From Definition~\ref{def:matrices_P_F_S}, we can obtain~\eqref{eq:recursion_S1}  by directly setting $i=N+1$.

The recursion~\eqref{eq:recursion_S2} can be derived by using Definition~\ref{def:matrices_P_F_S} to rewrite Definition~\ref{def:matrices_P_F_S} as follows:
\ifTwoColumn
\begin{align*}
  &S_i^{j} = \int_{\tau_i^{j}}^{T_\delta}\Phi(t,\tau_i^{j})^\top Q \Phi(t,\tau_i^{j})\mathrm{d}t + \Phi(T_\delta,\tau_i^{j})^\top E \Phi(T_\delta,\tau_i^{j})\\
  &= \begin{multlined}[t][0.85\columnwidth]\int_{\tau_i^{j}}^{\tau_{i}^{j+1}}\Phi(t,\tau_i^{j})^\top Q \Phi(t,\tau_i^{j})\mathrm{d}t \\+ \int_{\tau_{i}^{j+1}}^{T_\delta}\Phi(t,\tau_i^{j})^\top Q \Phi(t,\tau_i^{j})\mathrm{d}t + \Phi(T_\delta,\tau_i^{j})^\top E \Phi(T_\delta,\tau_i^{j})\end{multlined}\\
  &=\begin{multlined}[t][0.9\columnwidth]\int_{\tau_i^{j}}^{\tau_{i}^{j+1}}\Phi(t,\tau_i^{j})^\top Q \Phi(t,\tau_i^{j})\mathrm{d}t+\\ +\Phi(\tau_{i}^{j+1},\tau_i^{j})^\top\bigg(\int_{\tau_{i}^{j+1}}^{T_\delta}\Phi(t,\tau_{i}^{j+1})^\top Q \Phi(t,\tau_{i}^{j+1})\mathrm{d}t \\ + \Phi(T_\delta,\tau_{i}^{j+1})^\top E \Phi(T_\delta,\tau_{i}^{j+1})\bigg)\Phi(\tau_{i}^{j+1},\tau_i^{j})
  \end{multlined}\\
  &=\begin{multlined}[t][0.85\columnwidth]\int_{\tau_i^{j}}^{\tau_{i}^{j+1}}\Phi(t,\tau_i^{j})^\top Q \Phi(t,\tau_i^{j})\mathrm{d}t \\ + \Phi(\tau_{i}^{j+1},\tau_i^{j})^\top S_{i}^{j+1} \Phi(\tau_{i}^{j+1},\tau_i^{j})\end{multlined}\\
  &=\begin{multlined}[t][0.85\columnwidth]\int_{0}^{\delta_{i}^{j}}\Phi(\eta+\tau_i^{j},\tau_i^{j})^\top Q \Phi(\eta+\tau_i^{j},\tau_i^{j})\mathrm{d}\eta \\+ \Phi(\tau_{i}^{j+1},\tau_i^{j})^\top S_{i}^{j+1} \Phi(\tau_{i}^{j+1},\tau_i^{j})\end{multlined}\\
  &=\int_{0}^{\delta_{i}^{j}} e^{A_i^{j\top} \eta}  Q e^{A_i^{j}\eta} \mathrm{d}\eta + e^{A_i^{j\top} \delta_{i}^{j}}S_{i}^{j+1}  e^{A_i^{j\top} \delta_{i}^{j}}\\
  &= M_i^j + \mathcal{E}_{i}^{j\top} S_{i}^{j+1} \mathcal{E}_i^{j}.
\end{align*}
\else
\begin{align*}
  &S_i^{j} = \int_{\tau_i^{j}}^{T_\delta}\Phi(t,\tau_i^{j})^\top Q \Phi(t,\tau_i^{j})\mathrm{d}t + \Phi(T_\delta,\tau_i^{j})^\top E \Phi(T_\delta,\tau_i^{j})\\
  &= \int_{\tau_i^{j}}^{\tau_{i}^{j+1}}\Phi(t,\tau_i^{j})^\top Q \Phi(t,\tau_i^{j})\mathrm{d}t+ \int_{\tau_{i}^{j+1}}^{T_\delta}\Phi(t,\tau_i^{j})^\top Q \Phi(t,\tau_i^{j})\mathrm{d}t + \Phi(T_\delta,\tau_i^{j})^\top E \Phi(T_\delta,\tau_i^{j})\\
  &= \begin{multlined}[t]\int_{\tau_i^{j}}^{\tau_{i}^{j+1}}\Phi(t,\tau_i^{j})^\top Q \Phi(t,\tau_i^{j})\mathrm{d}t+\\ +\Phi(\tau_{i}^{j+1},\tau_i^{j})^\top\bigg(\int_{\tau_{i}^{j+1}}^{T_\delta}\Phi(t,\tau_{i}^{j+1})^\top Q \Phi(t,\tau_{i}^{j+1})\mathrm{d}t + \Phi(T_\delta,\tau_{i}^{j+1})^\top E \Phi(T_\delta,\tau_{i}^{j+1})\bigg)\Phi(\tau_{i}^{j+1},\tau_i^{j})\end{multlined}\\
  &=\int_{\tau_i^{j}}^{\tau_{i}^{j+1}}\Phi(t,\tau_i^{j})^\top Q \Phi(t,\tau_i^{j})\mathrm{d}t+ \Phi(\tau_{i}^{j+1},\tau_i^{j})^\top S_{i}^{j+1} \Phi(\tau_{i}^{j+1},\tau_i^{j})\\
  &=\int_{0}^{\delta_{i}^{j}}\Phi(\eta+\tau_i^{j},\tau_i^{j})^\top Q \Phi(\eta+\tau_i^{j},\tau_i^{j})\mathrm{d}\eta + \Phi(\tau_{i}^{j+1},\tau_i^{j})^\top S_{i}^{j+1} \Phi(\tau_{i}^{j+1},\tau_i^{j})\\
  &=\int_{0}^{\delta_{i}^{j}} e^{A_i^{j\top} \eta}  Q e^{A_i^{j}\eta} \mathrm{d}\eta + e^{A_i^{j\top} \delta_{i}^{j}}S_{i}^{j+1}  e^{A_i^{j\top} \delta_{i}^{j}}\\
  &= M_i^j + \mathcal{E}_{i}^{j\top} S_{i}^{j+1} \mathcal{E}_i^{j}.
\end{align*}
\fi
In the second equality we split the integral in two parts.
In the third equality we bring the matrices $\Phi(\tau_{i}^{j+1},\tau_i^{j})$ and $\Phi(\tau_{i}^{j+1},\tau_i^{j})^\top$ outside the integrals ince they do not depend on $t$.
In the fourth equality we apply the Definition~\ref{def:matrices_P_F_S} to obtain $S_{i}^{j+1}$.
In the fifth equality we applied a change of variables $\eta = t - \tau_i^j$.
In the sixth equality we rewrite the transition matrices using matrix exponentials.
Finally, we apply Definition~\ref{def:auxiliary_matrices_sto} to complete the proof.

\hfill \qedsymbol

\ifTwoColumn
  \balance
\fi

\ifTechReport
\fi

  \bibliographystyle{IEEEtran}
  \bibliography{bibliography}

  %





  \ifTechReport \else
  \begin{IEEEbiography}[{\includegraphics[width=1in,height=1.25in,clip,keepaspectratio]{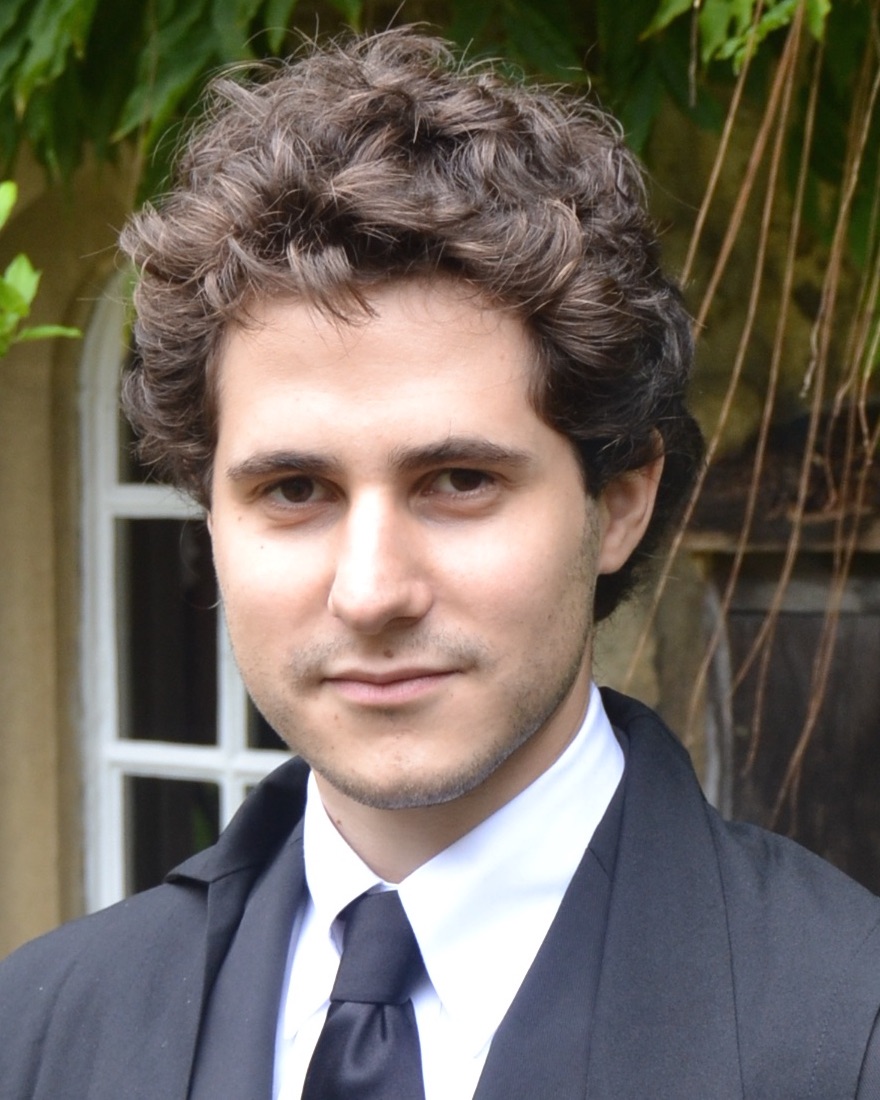}}]{Bartolomeo Stellato} (S'16)
  received the B.Sc. degree in automation engineering from Politecnico di Milano, Milano, Italy in 2012 and the M.Sc. degree in robotics, systems and control from ETH Z\"{u}rich, Z\"{u}rich, Switzerland in 2014. He is pursuing his D.Phil. (Ph.D.) degree in control systems and optimization at the University of Oxford under the supervision of Professor Paul J.\ Goulart.

  His current research interests are integer optimal control of fast dynamical systems, model predictive control and discrete optimization.
  \end{IEEEbiography}

  \begin{IEEEbiography}[{\includegraphics[width=1in,height=1.25in,clip,keepaspectratio]{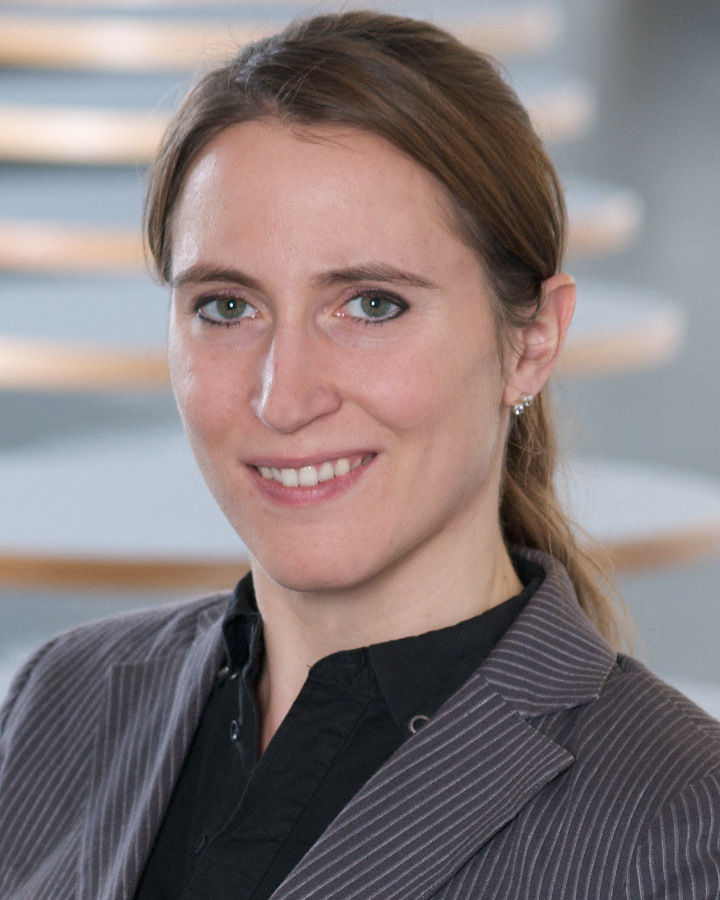}}]{Sina Ober-Bl\"{o}baum} received her Dipl.\ Math.\ in 2004 and her Ph.D.\ degree in applied mathematics in 2008 both from the University of Paderborn, Germany.

  She has been a Postdoctoral Scholar at California Institute of Technology, CA, USA, from 2008 to 2009. From 2009 to 2015 she has been a Junior Professor at the Department of Mathematics at the University of Paderborn. She is currently Associate Professor of control engineering at the Department of Engineering Science and a Tutorial Fellow in Engineering at Harris Manchester College, University of Oxford, Oxford, U.K.. Her research focus lies in the development and analysis of structure-preserving simulation and optimal control methods for mechanical, electrical and hybrid systems, with a wide range of application areas including astrodynamics, drive technology and robotics.
  \end{IEEEbiography}

  \begin{IEEEbiography}[{\includegraphics[width=1in,height=1.25in,clip,keepaspectratio]{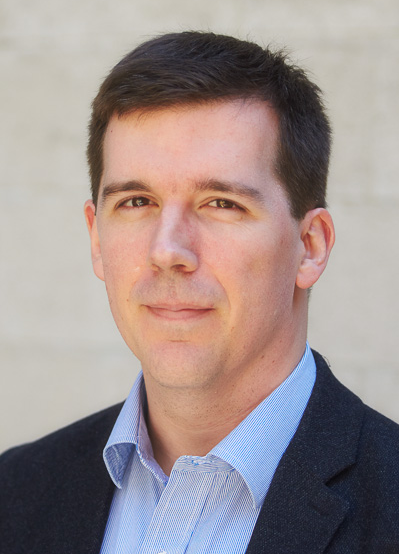}}]{Paul J.\ Goulart} (M'04)
  received the B.Sc. and M.Sc. degrees in aeronautics and astronautics from the Massachusetts Institute of Technology (MIT), Cambridge, MA, USA.  He was selected as a Gates Scholar at the University of Cambridge, Cambridge, U.K., where he received the Ph.D. degree in control engineering in 2007.

  From 2007 to 2011, he was a Lecturer in control systems in the Department of Aeronautics, Imperial
  College London, and from 2011 to 2014 a Senior Researcher in the Automatic Control Laboratory,
  ETH Z\"{u}rich. He is currently an Associate Professor in the Department of Engineering Science and Tutorial Fellow, St. Edmund Hall, University of Oxford, Oxford, U.K.. His research interests include model predictive control, robust optimization, and control of fluid flows.
  \end{IEEEbiography}

  \fi



\end{document}